\documentclass[11pt,oneside,leqno,english]{amsart}
\usepackage{amstext}
\usepackage{amsthm}
\usepackage[osf]{libertineRoman}
\usepackage[osf,scaled=0.95]{biolinum}

\usepackage[libertine]{newtxmath}
\usepackage[latin9]{inputenc}
\usepackage[a4paper]{geometry}
\usepackage{bm}
\usepackage{graphicx}
\usepackage{setspace}
\makeatletter
\numberwithin{equation}{section}
\numberwithin{figure}{section}
\theoremstyle{plain}
\newtheorem{thm}{\protect\theoremname}
\theoremstyle{plain}
\newtheorem{lem}[thm]{\protect\lemmaname}
\theoremstyle{plain}
\newtheorem{prop}[thm]{\protect\propositionname}
\theoremstyle{remark}
\newtheorem{rem}[thm]{\protect\remarkname}
\theoremstyle{plain}
\newtheorem*{thm*}{\protect\theoremname}

\AtBeginDocument{
  
}\usepackage{babel}

\makeatother

\usepackage{babel}
\providecommand{\lemmaname}{Lemma}
\providecommand{\propositionname}{Proposition}
\providecommand{\remarkname}{Remark}
\providecommand{\theoremname}{Theorem}

\begin{document}
\title[instantaneous control of a melting problem]{ An instantaneous semi-Lagrangian approach for boundary control of
a melting problem}
\author{Youness Mezzan \& moulay hicham tber}
\address{Cadi Ayyad University, Department of Mathematics,
Av. Abdelkarim Elkhattabi, Marrakech, Morocco.}
\begin{abstract}
In this paper, a sub-optimal boundary control strategy for a free
boundary problem is investigated. The model is described by a non-smooth
convection-diffusion equation. The control problem is addressed by
an instantaneous strategy based on the characteristics method. The
resulting time independent control problems are formulated as function
space optimization problems with complementarity constraints. At each
time step, the existence of an optimal solution is proved and first-order
optimality conditions with regular Lagrange multipliers are derived
for a penalized-regularized version. The performance of the overall
approach is illustrated by numerical examples.\\
 
\end{abstract}

\maketitle

\section{Introduction}

Heat transfer processes involving phase change are relevant to many
engineering disciplines including casting of metals, thermal storage,
power systems, micro-electronics, etc \cite{Dhir}. Enhancing the
thermal performance of systems using such processes requires a proper
control of the temperature profile and the associated phase change
interface.\\
 Our motivation in this paper is to design an optimization strategy
for a melting process that might be affected by a convection in the
liquid phase. We focus on two-phase materials with sharp interface
and we adopt a single domain approach where the Stefan condition is
automatically satisfied across the free boundary. More precisely we
consider a source-based method in which the total enthalpy is split
into a specific heat and a latent heat acting as a source term in
the energy equation \cite{Saadi}. Our goal is to control the temperature
profile using the heat flux on a part of the boundary. This task is,
nevertheless, quite challenging even for simple geometries. In fact,
the liquid-solid free boundary changes sharply with respect to the
temperature. Furthermore, from a numerical point of view, solutions
may exhibits non-physical oscillations for convection dominated flows.
Finally, the related optimal control problem is very demanding in
terms of computational time and storage.

Optimal control problems in the context of Stefan-like models have
attracted a lot of attention since the eighties of the last century.
We refer, in particular, to the monograph \cite{Goldman} and the
references there. However, most used models were generally based on
simplified assumptions on the free boundary, and therefore describe
roughly the phase-change process. Subsequent studies \cite{Hinze-Ziegenberg-1,Hinze-Ziegenberg-2,Bernauer,Baran}
have considered two-phase Stefan problems with a focus on numerical
aspects. Recently, some existence and differentiability results are
established in \cite{Abdulla-1,Abdulla-2,Abdulla-3} for one-dimensional
problems. \\
 To accommodate the problem, our strategy here exploits a semi-Lagrangian
scheme \cite{Pironneau-Ziebenbalg} in the context of an instantaneous
control approach \cite{Choi-et-al,Choi-Hinze-Kunisch}. The time derivative
and the convection terms are combined as a directional derivative
along the characteristics. We show that the time-discrete state equation
satisfies a maximum principle. Then, at each time step we cast the
time-discrete optimal control problem - which only depends on the
state at the previous time - as an optimization problem with a complemantarity
constraint between the temperature and solid fraction. However, due
to the structure of the feasible set, standard numerical algorithms
can't be applied directly to solve such optimization problems (see
for instance \cite{Hintermuller et al.}). Here, we propose a regularization-penalization
technique where we first regularize the constraint on the temperature
variable then we incorporate the related complementarity into the
objective functional via an $\ell_{1}-$penalty approach \cite{Loebhard}.
For the resulting regularized-penalized problems we show an existence
and consistency result and further we derive first-order necessary
optimality conditions that enjoy regular Lagrange multipliers. The
over all approach leads, naturally, to sub-optimal solutions. Nevertheless,
a good performance is achieved in the numerical experiments.

\section{State equation}

\subsection*{Mathematical model}

We consider the melting of a finite slab of a pure substance. The
model is described by the non-dimensional source-based Stefan equation
\[
\dfrac{\partial y}{\partial t}+\overrightarrow{v}\cdot\nabla y-\nabla\cdot\left(k\,\nabla y\right)=\dfrac{\partial}{\partial t}\xi+\overrightarrow{v}\cdot\nabla\xi\qquad\text{in }\Omega\times\left(t_{0},\,t_{f}\right),
\]
where $\kappa=\kappa(x,\,t)$ is the thermal conductivity and $\overrightarrow{v}=\overrightarrow{v}(x,\,t)$
is a convection velocity. The solid fraction $\xi=\xi\left(x,\,t\right)$
and temperature distribution $y=y(x,\,t),$ are related through the
relation 
\[
\xi\in\,\mathcal{H}(y):=\left\{ \begin{array}{ll}
0 & \quad\textrm{ if }y>0,\\
\left[0,\,1\right] & \quad\textrm{ if }y=0,\\
1 & \quad\textrm{ if }y<0.
\end{array}\right.
\]

Here the phase-change processes is assumed to be isothermal. The model
domain $\Omega$ is an open bounded of $\mathbb{R}^{n}$ $(n=1,\,2)$
with a smooth boundary $\Gamma$ corresponding to both solid and liquid
regions (see Fig.~\ref{fig-Pb-Config}). On $\Gamma$ we distinguish
three parts: the system is insulated on $\Gamma_{N},$ a fixed temperature
$y_{D}=0$ is maintained on $\Gamma_{D}$ and a non-negative heat
flux control $u=u(x,\,t)$ is applied on $\Gamma_{C}.$ The substance
is initially at the melting/freezing point 
\[
y(x,\,t_{0})=0,\qquad\qquad\xi(x,\,t_{0})=\xi_{0}(x)\in\left[0,\,1\right]\qquad\qquad\text{ for }x\in\Omega.
\]
The complete model equation reads 
\[
\tag{\ensuremath{\mathcal{M}^{t}}}\left\{ \begin{array}{ll}
\dfrac{\partial y}{\partial t}+\overrightarrow{v}\cdot\nabla y-\nabla\cdot\left(\kappa\,\nabla y\right)=\dfrac{\partial}{\partial t}\xi+\overrightarrow{v}\cdot\nabla\xi & \text{in }\Omega\times\left(t_{0},\,t_{f}\right),\\
\xi\in\,\mathcal{H}(y) & \text{ in }\Omega\times\left(t_{0},\,t_{f}\right),\\
\dfrac{\partial y}{\partial n}=0 & \text{in }\Gamma_{N}\times\left(t_{0},\,t_{f}\right),\\
\dfrac{\partial y}{\partial n}=u & \text{in }\Gamma_{C}\times\left(t_{0},\,t_{f}\right),\\
y=0 & \text{ in }\Gamma_{D}\times\left(t_{0},\,t_{f}\right),\\
y\left(t_{0}\right)=0,\,\xi\left(t_{0}\right)=\xi_{0} & \text{ in }\Omega.
\end{array}\right.
\]
\begin{figure}
\centering{}\includegraphics[width=8cm,height=5cm]{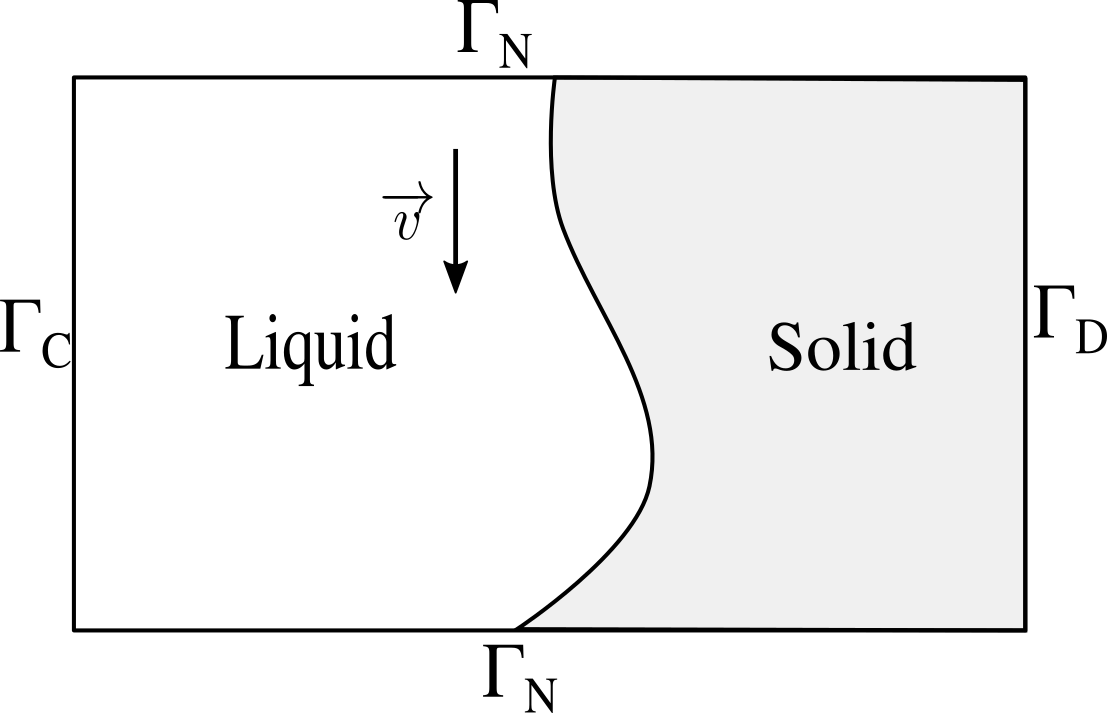}\caption{\label{fig-Pb-Config}Problem configuration}
\end{figure}

\subsection*{Time discretization}

Due to the hyperbolic character of the state equation, the numerical
solutions may exhibit undesired oscillations for dominated convection
terms. One approach to deal with this issue consists in writing $\dfrac{\partial\phi}{\partial t}+\overrightarrow{v}\cdot\nabla\phi$
as $\dfrac{D\phi}{Dt}$ the material derivative of a given function
$\phi$ in the direction of $\overrightarrow{v}.$ The corresponding
characteristic curves are defined by 
\[
\left\{ \begin{array}{c}
\dfrac{dX(x,\,t;\,s)}{ds}=\overrightarrow{v}\left(x,\,t\right),\\
X(x,\,t;\,t)=x,
\end{array}\right.
\]
with $X(x,t;s)$ being the position of a particle at time $s,$ which
was at $x$ at time $t$. \\
 Now for a given uniform time step size $\tau=\dfrac{\left(t_{f}-t_{0}\right)}{N}>0,$
we can get an approximate value of $X$ at time $t^{n-1}=t_{0}+\left(n-1\right)\tau$
by 
\[
X^{n}(\,x\,):=X(x,\,t^{n};\,t^{n-1})=x-\tau\,\overrightarrow{v}\left(x,\,t^{n}\right)\qquad n=1,\dots N.
\]

Using a fully-implicit scheme, we obtain the semi-discrete form of
$\left(\mathcal{M}^{t}\right)$

\[
\tag{\ensuremath{\mathcal{M}^{\tau}}}\left.\begin{array}{ll}
y^{n}-\tau\,\nabla\cdot\left(\kappa^{n}\,\nabla y^{n}\right)=\xi^{n}+\overline{y}^{n-1}-\overline{\xi}^{n-1} & \text{in }\Omega,\\
\xi^{n}\in\,\mathcal{H}(y^{n}) & \text{ in }\Omega,\\
\dfrac{\partial y^{n}}{\partial n}=0 & \text{on }\Gamma_{N},\\
\dfrac{\partial y^{n}}{\partial n}=u^{n} & \text{on }\Gamma_{C},\\
y^{n}=0 & \text{ on }\Gamma_{D},\\
\xi^{0}=\xi_{0},\quad y^{0}=0 & \text{ in }\Omega,
\end{array}\right\} n=1,\dots N.
\]
where $\phi^{n}\left(\cdot\right):=\phi\left(\cdot,\,t^{n}\right)$
and $\overline{\phi}^{n-1}:=\phi^{n-1}\circ X^{n}.$ To avoid technical
difficulties, it is assumed that $X^{n}$ maps $\Omega$ to itself.
Formulation $\left(\mathcal{M}^{\tau}\right)$ has the advantage of
not being restricted by a CFL condition and large time steps may be
used \cite{Pironneau}.

\subsection*{Variational Formulation}

In the following standard notations for Lebesgue and Sobolev spaces
are employed (see e.g. \cite[Chap. 5]{Evans}). The $L^{2}\left(\Omega\right)$
norm for either vector-valued or real-valued functions is denoted
by $\|\cdot\|.$ The $L^{2}\left(\Gamma_{C}\right)$ norm is specified
by $\|\cdot\|_{\Gamma_{C}}.$ To define a variational formulation
for the semi-discrete problem we introduce the space 
\[
\mathcal{V}:=\left\{ \phi\in H^{1}\left(\Omega\right)\::\:\phi=0\text{ on }\Gamma_{D}\right\} 
\]
endowed with the $H^{1}\left(\Omega\right)$ norm $\|\,\cdot\,\|_{H^{1}\left(\Omega\right)}.$
\\
 At a specific time step $t^{n}$ the variational formulation of the
semi-discrete state equation consists in finding $\left(y^{n},\,\xi^{n}\right)\in\mathcal{V}\times L^{2}\left(\Omega\right)$
such that 
\[
\tag{\ensuremath{\mathcal{WF}^{n}}}\left\{ \begin{array}{ll}
Ay^{n}=\xi^{n}+Bu^{n}+\overline{y}^{n-1}-\overline{\xi}^{n-1}\quad & \text{in }\mathcal{V}^{\prime},\\
\xi^{n}\in\mathcal{H}(y^{n}) & \quad\text{a.e. in }\Omega,
\end{array}\right.
\]
for given $u^{n}\in L^{2}\left(\Gamma_{C}\right),$ $\xi^{n-1}\in L^{2}\left(\Omega\right)$
and $y^{n-1}\in\mathcal{V}.$ $B:L^{2}\left(\Gamma_{C}\right)\mapsto\mathcal{V}^{\prime}$
and $A:\mathcal{V}\mapsto\mathcal{V}^{\prime}$ stand for the linear
bounded operator defined by 
\begin{align*}
\left\langle Bv,\,\phi\right\rangle  & :=\tau\left(v,\,\gamma_{0}\phi\right)_{\Gamma_{C}}\qquad\forall\left(v,\,\phi\right)\in L^{2}\left(\Gamma_{C}\right)\times\mathcal{V},\\
\left\langle A\psi,\,\phi\right\rangle  & :=\left(\psi,\,\phi\right)+\tau\left(\kappa^{n}\,\nabla\psi,\,\nabla\phi\right)\qquad\forall\left(\psi,\,\phi\right)\in\mathcal{V}\times\mathcal{V},
\end{align*}
where $\left\langle \cdot\,,\,\cdot\right\rangle $ is the pairing
between $\mathcal{V}$ and its dual $\mathcal{V}^{\prime}.$ The inner
products in $L^{2}\left(\Omega\right)$ and $L^{2}\left(\Gamma_{C}\right)$
are indicated by $\left(\cdot\,,\,\cdot\right)$ and $\left(\cdot\,,\,\cdot\right)_{\Gamma_{C}}$
respectively. $\gamma_{0}$ is the trace operator in $H^{1}\left(\Omega\right)$
and $\kappa^{n}\in L^{\infty}\left(\Omega\right)$ is such that the
$A$ operator is uniformly coercive with a constant $\underline{\kappa}.$

Regarding the solvability of $\left(\mathcal{WF}^{n}\right)$ we state
the following theorem whose proof is deferred to Appendix A. 
\begin{thm}
\label{thm-1}Let $u^{n}\in L^{2}\left(\Gamma_{C}\right),$ $y^{n-1}\in\mathcal{V}$
and $\xi^{n-1}\in L^{2}\text{\ensuremath{\left(\Omega\right)}}$ such
that $u^{n}\geq0$ a.e. in $\Gamma_{C},$ $y^{n-1}\geq0$ a.e. in
$\Omega$ and $0\leq\xi^{n-1}\leq1$ a.e. in $\Omega.$ Problem $\left(\mathcal{WF}^{n}\right)$
has one and only solution $\left(y^{n},\,\xi^{n}\right)\in\mathcal{V}\times L^{2}\left(\Omega\right)$
that is given by the solution of 
\[
\tag{\ensuremath{\mathcal{CS}^{n}}}\left\{ \begin{array}{ll}
Ay^{n}=\xi^{n}+Bu^{n}+\overline{y}^{n-1}-\overline{\xi}^{n-1} & \quad\text{in }\mathcal{V}^{\prime},\\
y^{n}\geq0,\qquad\xi^{n}\geq0 & \text{a.e. in }\Omega,\\
\left(y^{n},\,\xi^{n}\right)=0.
\end{array}\right.
\]
\end{thm}

\section{Sub-optimal control problem}

In the following we aim to steer the system to a desired configuration,
by acting on the heat flux $u$ at the boundary $\Gamma_{C}.$ We
adopt an instantaneous optimal control concept: at each time step
$t^{n},$ given the previous temperature and solid fraction profiles
$y^{n-1}$ and $\xi^{n-1},$ we solve a time-independent optimal control
problem. Regarding the previous theorem we consider the following
PDE-constrained optimization problems 
\[
\tag{\ensuremath{\mathcal{O}^{n}}}\left\{ \begin{array}{ll}
\min & J(y^{n},\,\xi^{n},\,u^{n})=\dfrac{1}{2}\|y^{n}-y_{d}^{n}\|_{H^{1}\left(\Omega\right)}^{2}+\dfrac{1}{2}\|\xi^{n}-\xi_{d}^{n}\|^{2}+\dfrac{\nu}{2}\|u^{n}\|_{\Gamma_{C}}^{2},\\
\mbox{over } & \left(y^{n},\,\xi^{n},\,u^{n}\right)\in\mathcal{V}\times L^{2}\left(\Omega\right)\times L^{2}\left(\Gamma_{C}\right),\\
\text{s.t.} & Ay^{n}=\xi^{n}+Bu^{n}+\overline{y}^{n-1}-\overline{\xi}^{n-1}\qquad\text{ in }\mathcal{V}^{\prime},\\
 & y^{n}\geq0\qquad\qquad\textrm{ a.e. in }\Omega,\\
 & \xi^{n}\geq0\qquad\qquad\textrm{ a.e. in }\Omega,\\
 & u^{n}\geq0\qquad\qquad\textrm{ a.e. in }\Gamma_{C},\\
 & \left(\xi^{n},\,y^{n}\right)=0.
\end{array}\right.
\]
where $y_{d}^{n}\in H^{1}\left(\Omega\right)$ and $\xi_{d}^{n}\in L^{2}\left(\Omega\right)$
correspond to a desired state at time $t^{n}$ and $\nu$ is a regularization
parameter.

The next lemma serves as a tool to establish some results of this
paper. Its proof is straightforward. 
\begin{lem}
\label{AuxLemma-1} Let $\left(y_{k},\,\xi_{k},\,u_{k}\right){}_{k\in\mathbb{N}}$
be a sequence in $\mathcal{V}\times L^{2}\left(\Omega\right)\times L^{2}\left(\Gamma_{C}\right)$
such that $\left(\xi_{k},\,u_{k}\right)_{k\in\mathbb{N}}$ is bounded
in $L^{2}\left(\Omega\right)\times L^{2}\left(\Gamma_{C}\right)$
and 
\begin{align}
Ay_{k}=Bu_{k}+\xi_{k}+\overline{y}^{n-1}-\overline{\xi}^{n-1} & \qquad\text{in }\mathcal{V}^{\prime},\label{AuxLemmaEq1-1-1}\\
y_{k}\geq0,\quad\xi_{k}\geq0 & \qquad\text{a.e. in }\Omega,\label{AuxLemmaEq3Prime-1-1}\\
u_{k}\geq0\quad & \text{\,a.e. in }\Gamma_{C}.
\end{align}
Then, there exists a sub-sequence still denoted by $(y_{k},\,u_{k},\,\xi_{k})_{k\in\mathbb{N}}$
such that 
\begin{eqnarray}
u_{k} & \rightharpoonup & u\quad\text{ in }L^{2}\left(\Gamma_{C}\right),\label{ConvUGlobal-3}\\
\xi_{k} & \rightharpoonup & \xi\quad\text{ in }L^{2}\left(\Omega\right),\label{ConvLambdaGlobal-3}\\
y_{k} & \longrightarrow & y\quad\text{ in }L^{2}\left(\Omega\right),\label{ConvYGlobal-H1-3}\\
y_{k} & \rightharpoonup & y\quad\text{ in }\mathcal{V},\label{ConvYGlobal-L2-3}
\end{eqnarray}
with $(y,\,\xi,\,u)$ being an element of $\mathcal{V}\times L^{2}\left(\Omega\right)\times L^{2}\left(\Gamma_{C}\right)$
satisfying 
\begin{align}
Ay=Bu+\xi+\overline{y}^{n-1}-\overline{\xi}^{n-1} & \qquad\text{in }\mathcal{V}^{\prime},\label{AuxLemmaLimEq1}\\
y\geq0,\quad\xi\geq0 & \qquad\text{a.e. in }\Omega,\label{AuxLemmaLimEq2}\\
u\geq0 & \qquad\text{a.e. in }\Gamma_{C}.\label{AuxLemmaLimEq3}
\end{align}
Further 
\begin{align}
\underset{k\to\infty}{\lim}\left(\xi_{k},\,y_{k}\right)= & \left(\xi,\,y\right).\label{AuxLemmaLimEq4}
\end{align}
In particular, if $\left(\xi_{k},\,y_{k}\right)=0$ then $\left(\xi,\,y\right)=0.$ 
\end{lem}

\begin{thm}
Problem $\left(\mathcal{O}^{n}\right)$ has at least one solution. 
\end{thm}

\begin{proof}
Let $\left(y_{k}^{n},\,\xi_{k}^{n},\,u_{k}^{n}\right)_{k\in\mathbb{N}}\in\mathcal{V}\times L^{2}\left(\Omega\right)\times L^{2}\left(\Gamma_{C}\right)$
be a minimizing sequence of $J$ over the feasible set of $\left(\mathcal{O}^{n}\right).$
Then $\left(\xi_{k}^{n},\,u_{k}^{n}\right)_{k\in\mathbb{N}}$ is bounded
in $L^{2}\left(\Omega\right)\times L^{2}\left(\Gamma_{C}\right).$
From Lemma \ref{AuxLemma-1}, there exists a feasible element $(y^{n},\,\xi^{n},\,u^{n})$
such that up to a sub-sequence $\left(y_{k}^{n},\,\xi_{k}^{n},\,u_{k}^{n}\right)_{k\in\mathbb{N}}$
converges weakly to $\left(y^{n},\,\xi^{n},\,u^{n}\right)$ in $\mathcal{V}\times L^{2}\left(\Omega\right)\times L^{2}\left(\Gamma_{C}\right).$
It is immediate to verify that $J$ is weakly lower semi-continuous
which proves that $\left(y^{n},\,\xi^{n},\,u^{n}\right)$ is a solution
of $\left(\mathcal{O}^{n}\right).$ 
\end{proof}

\section{Penalized-regularized optimal control problem}

In this section we propose a function space approach to solve the
problem $\left(\mathcal{O}^{n}\right).$ Inspired by \cite{Loebhard},
we process a series of sub-problems $\left(\mathcal{O}_{\gamma}^{n}\right)_{\gamma>0}$
defined by 
\begin{equation}
\tag{\ensuremath{\mathcal{O}_{\gamma}^{n}}}\left\{ \begin{array}{ll}
\min & J_{\gamma}(y^{n},\,\xi^{n},\,u^{n}):=J(y^{n},\,\xi^{n},\,u^{n})+\gamma\left(\xi^{n},\,y^{n}+\varepsilon_{\gamma}\xi^{n}\right),\\
\mbox{over } & \left(y^{n},\,\xi^{n},\,u^{n}\right)\in\mathcal{V}\times L^{2}\left(\Omega\right)\times L^{2}\left(\Gamma_{C}\right),\\
\text{s.t.} & Ay^{n}=Bu^{n}+\xi^{n}+\overline{y}^{n-1}-\overline{\xi}^{n-1}\qquad\text{ in }\mathcal{V}^{\prime},\\
 & y^{n}+\varepsilon_{\gamma}\xi^{n}\geq0\qquad\qquad\textrm{ a.e. in }\Omega,\\
 & \xi^{n}\geq0\qquad\qquad\textrm{ a.e. in }\Omega,\\
 & u^{n}\geq0\qquad\qquad\textrm{ a.e. in }\Gamma_{C},
\end{array}\right.\label{Pgamma}
\end{equation}
where $\gamma$ and $\varepsilon_{\gamma}$ are positive parameters
such that $\gamma\to\infty$ and $\varepsilon_{\gamma}\to0.$ More
precisely we assume that 
\[
\varepsilon_{\gamma}\gamma\underset{\gamma\to\infty}{\longrightarrow}0.
\]
The complementarity constraint will be increasingly satisfied by letting
$\gamma\to\infty$ which provide a path-following method for the solution
of the original control problems $\left(\ensuremath{\mathcal{O}^{n}}\right).$
Further we will show that Lagrange multipliers for $\left(\ensuremath{\mathcal{O}_{\gamma}^{n}}\right)$
are regular functions, so that using, for instance, a conform finite
elements discretization in numerical experiments is justified.

Here and in the following $C$ is a generic constant not depending
on $\gamma.$

\subsection*{Solvability and consistency of $\left(\mathcal{O}_{\gamma}^{n}\right)_{\gamma>0}$}
\begin{thm}
\label{ThmGlobSolPenPb} For every fixed $\gamma>0$ the penalized-regularized
problem $\left(\mathcal{O}_{\gamma}^{n}\right)$ has at least one
solution $\left(y_{\gamma}^{n},\,\xi_{\gamma}^{n},\,u_{\gamma}^{n}\right)$
in $\mathcal{V}\times L^{2}\left(\Omega\right)\times L^{2}\left(\Gamma_{C}\right).$
Furthermore, there exist $\left(y_{*}^{n},\,\xi_{*}^{n},\,u_{*}^{n}\right)$
in $\mathcal{V}\times L^{2}\left(\Omega\right)\times L^{2}\left(\Gamma_{C}\right)$
and a sub-sequence $\left(y_{\gamma}^{n},\,\xi_{\gamma}^{n},\,u_{\gamma}^{n}\right)_{\gamma>0}$
such that 
\begin{eqnarray}
u_{\gamma}^{n} & \rightharpoonup & u_{*}\quad\text{ in }L^{2}\left(\Gamma_{C}\right),\label{ConvUGammaGlobal}\\
\xi_{\gamma}^{n} & \rightharpoonup & \text{\ensuremath{\xi}}_{*}\quad\text{ in }L^{2}\left(\Omega\right),\label{ConvLambdaGammaGlobal}\\
y_{\gamma}^{n} & \longrightarrow & y_{*}^{n}\quad\text{ in }L^{2}\left(\Omega\right),\label{ConvYGammaGlobal}\\
y_{\gamma}^{n} & \rightharpoonup & y_{*}^{n}\quad\text{ in }\mathcal{V},\label{ConvYGammaGlobal-w}
\end{eqnarray}
and $\left(y_{*}^{n},\,\xi_{*}^{n},\,u_{*}^{n}\right)$ is a solution
of $\left(\mathcal{O}^{n}\right).$ 
\end{thm}

\begin{proof}
The existence of a solution $\left(y_{\gamma}^{n},\,\xi_{\gamma}^{n},\,u_{\gamma}^{n}\right)$
to $\left(\mathcal{O}_{\gamma}^{n}\right)$ follows from Lemma \ref{AuxLemma-1}
and $J_{\gamma}$ weak lower semi-continuity applied to a minimizing
sequence. We recall that 
\begin{align*}
J_{\gamma}\left(y_{\gamma}^{n},\,\xi_{\gamma}^{n},\,u_{\gamma}^{n}\right) & :=J\left(y_{\gamma}^{n},\,\xi_{\gamma}^{n},\,u_{\gamma}^{n}\right)+\gamma\left(\xi_{\gamma}^{n},\,y_{\gamma}^{n}+\varepsilon_{\gamma}\xi_{\gamma}^{n}\right),\\
 & =J\left(y_{\gamma}^{n},\,\xi_{\gamma}^{n},\,u_{\gamma}^{n}\right)+\gamma\left(\xi_{\gamma}^{n},\,y_{\gamma}^{n}\right)+\gamma\varepsilon_{\gamma}\|\xi_{\gamma}^{n}\|^{2}.
\end{align*}
On the other hand 
\begin{align}
J_{\gamma}\left(y_{\gamma}^{n},\,\xi_{\gamma}^{n},\,u_{\gamma}^{n}\right) & \leq J_{\gamma}\left(\tilde{y},\,\tilde{\xi},\,\tilde{u}\right),\label{PenRelProbGlobSolPropProof-Eq1}\\
 & \leq J\left(\tilde{y},\,\tilde{\xi},\,\tilde{u}\right)+\gamma\varepsilon_{\gamma}\|\tilde{\xi}\|^{2},\\
 & \leq J\left(\tilde{y},\,\tilde{\xi},\,\tilde{u}\right)+C\|\tilde{\xi}\|^{2},
\end{align}
for all $\left(\tilde{y},\,\tilde{\xi},\,\tilde{u}\right)$ in $\mathcal{F}^{n}.$
Notice that $\mathcal{F}^{n}\subseteq\mathcal{F}_{\gamma}^{n}$ for
all $\gamma>0$ with $\mathcal{F}^{n}$ and $\mathcal{F}_{\gamma}^{n}$
being the feasible sets of $\left(\mathcal{O}^{n}\right)$ and $\left(\mathcal{O}_{\gamma}^{n}\right)$
respectively. \\
 Therefore, there exists a constant $C$ not depending on $\gamma$
such that 
\begin{align}
\|\xi_{\gamma}^{n}\|\leq C,\qquad\|u_{\gamma}^{n}\|\leq C,\qquad0\leq\left(\xi_{\gamma}^{n},\,y_{\gamma}^{n}+\varepsilon_{\gamma}\xi_{\gamma}^{n}\right)\leq\dfrac{C}{\gamma}, & \qquad\forall\gamma>0.\label{EstimationsProofConvGamma}
\end{align}
Then, by Lemma \ref{AuxLemma-1}, there exist a sub-sequence still
denoted by $\left(y_{\gamma},\,\xi_{\gamma},\,u_{\gamma}\right)\in\mathcal{F}_{\gamma}$
and $\left(y_{*}^{n},\,\xi_{*}^{n},\,u_{*}^{n}\right)$ in $\mathcal{V}\times L^{2}\left(\Omega\right)\times L^{2}\left(\Gamma_{C}\right)$
such that \eqref{ConvUGammaGlobal}-\eqref{ConvYGammaGlobal-w} hold
and 
\begin{align*}
Ay_{*}^{n}=Bu_{*}^{n}+\xi_{*}^{n}+\overline{y}^{n-1}-\overline{\xi}^{n-1} & \qquad\text{in }\mathcal{V}^{\prime},\\
y_{*}^{n}\geq0 & \qquad\text{a.e. in }\Omega,\\
\xi_{*}^{n}\geq0 & \qquad\text{a.e. in }\Omega,\\
u_{*}^{n}\geq0 & \qquad\text{a.e. in }\Gamma_{C}.
\end{align*}
From \eqref{ConvLambdaGammaGlobal}, \eqref{ConvYGammaGlobal}, \eqref{EstimationsProofConvGamma}
and $\underset{\gamma\to\infty}{\lim}\varepsilon_{\gamma}=0$ we have
\[
\lim_{\gamma\to\infty}\left(\xi_{\gamma}^{n},\,y_{\gamma}^{n}+\varepsilon_{\gamma}\xi_{\gamma}^{n}\right)=\lim_{\gamma\to\infty}\varepsilon_{\gamma}\|\xi_{\gamma}^{n}\|^{2}+\lim_{\gamma\to\infty}\left(\xi_{\gamma}^{n},\,y_{\gamma}^{n}\right)=\left(\xi_{*}^{n},\,y_{*}^{n}\right).
\]
\eqref{EstimationsProofConvGamma} yields additionally that 
\[
\lim_{\gamma\to\infty}\left(\xi_{\gamma}^{n},\,y_{\gamma}^{n}+\varepsilon_{\gamma}\xi_{\gamma}^{n}\right)=0.
\]
Hence, $\left(\xi_{*}^{n},\,y_{*}^{n}\right)=0$ and $\left(y_{*}^{n},\,\xi_{*}^{n},\,u_{*}^{n}\right)\in\mathcal{F}^{n}.$
\\
 Now from the weak lower semi-continuity of $J$ we have 
\begin{align*}
J\left(y_{*}^{n},\,\xi_{*}^{n},\,u_{*}^{n}\right) & \leq\liminf_{\gamma\to\infty}J\left(y_{\gamma}^{n},\,\xi_{\gamma}^{n},\,u_{\gamma}^{n}\right).
\end{align*}
Since $J\leq J_{\gamma},$ $\mathcal{F}^{n}\subseteq\mathcal{F}_{\gamma}^{n}$
and $\varepsilon_{\gamma}\gamma\underset{\gamma\to\infty}{\longrightarrow}0$
it follows that 
\begin{align*}
J\left(y_{*}^{n},\,\xi_{*}^{n},\,u_{*}^{n}\right) & \leq\liminf_{\gamma\to\infty}J_{\gamma}\left(y_{\gamma}^{n},\,\xi_{\gamma}^{n},\,u_{\gamma}^{n}\right),\\
 & \leq\liminf_{\gamma\to\infty}J_{\gamma}\left(\tilde{y},\,\tilde{\xi},\,\tilde{u}\right),\\
 & \leq J\left(\tilde{y},\,\tilde{\xi},\,\tilde{u}\right)+\lim_{\gamma\to\infty}\gamma\varepsilon_{\gamma}\|\tilde{\xi}\|^{2},\\
 & \leq J\left(\tilde{y},\,\tilde{\xi},\,\tilde{u}\right),
\end{align*}
for any $\left(\tilde{y},\,\tilde{\xi},\,\tilde{u}\right)$ in $\mathcal{F}.$\\
 Consequently, $\left(y_{*}^{n},\,\xi_{*}^{n},\,u_{*}^{n}\right)$
is a solution to the limit optimal control problem $\left(\mathcal{O}^{n}\right).$ 
\end{proof}

\subsection*{First order optimality conditions for $\left(\mathcal{O}_{\gamma}^{n}\right)_{\gamma>0}$}

In order to derive the first order optimality system for the regularized-penalized
problems $\left(\mathcal{O}_{\gamma}^{n}\right)_{\gamma>0}$ we check
the Zowe-Kurcyusz constraints qualification \cite{ZoweKurcyusz-1979,Troeltzsch}
which we recall in Appendix B. In our contest it requires the existence
of $\left(c_{y},\,c_{\xi},\,c_{u},\,\zeta,\,\lambda\right)$ in $\mathcal{V}\times L^{2}\left(\Omega\right)\times L^{2}\left(\Gamma_{C}\right)\times L^{2}\left(\Omega\right)\times\mathbb{R}$
such that the following system holds

\begin{align*}
\tag{\ensuremath{\mathcal{CQ}}}\left\{ \begin{alignedat}{1}Ac_{y}=Bc_{u}+c_{\xi}+z^{\prime} & \text{ in }\mathcal{V}^{\prime},\\
c_{y}+\varepsilon_{\gamma}c_{\xi}-\zeta+\lambda\left(y_{\gamma}^{n}+\varepsilon_{\gamma}\xi_{\gamma}^{n}\right)=z & \textrm{ a.e. in }L^{2}\left(\Omega\right),\\
c_{\xi}\geq0,\qquad\zeta\geq0 & \text{ a.e. in }L^{2}\left(\Omega\right),\\
c_{u}\geq0 & \text{ a.e. in }L^{2}\left(\Gamma_{C}\right),\\
\lambda\ge & 0,
\end{alignedat}
\right.
\end{align*}
for a given $\left(z^{\prime},\,z\right)\in\mathcal{V}^{\prime}\times L^{2}\left(\Omega\right).$
First, we pose 
\begin{align*}
\lambda & =0,\qquad c_{u}=0,\qquad c_{\xi}=c_{\xi,1}+c_{\xi,2},\qquad\zeta=\zeta_{1}+\zeta_{2}
\end{align*}
with 
\[
c_{\xi,2}=\dfrac{1}{\varepsilon_{\gamma}}\max\left(0,\,z\right),\qquad\zeta_{2}=\max\left(0,\,-z\right).
\]
Then, we choose $\zeta_{1}\in\mathcal{V}$ such that the system 
\begin{align}
A\zeta_{1} & =c_{\xi,2}+z^{\prime}+\Lambda\qquad\text{in }\mathcal{V}^{\prime},\label{ZoweIn1-1}\\
\Lambda\geq0,\qquad\zeta_{1}\geq0,\qquad & <\Lambda,\,\zeta_{1}>=0,\label{ZoweIn2-1}
\end{align}
holds for some $\Lambda\in\mathcal{V}^{\prime}.$ We mention that
\eqref{ZoweIn1-1}-\eqref{ZoweIn2-1} is well-posed by the theory
of variational inequalities \cite{Stampacchia}. Finally we assign
to $c_{\xi,1}$ the solution of the following elliptic partial differential
equation 
\begin{align}
\varepsilon_{\gamma}Ac_{\xi,1}+c_{\xi,1} & =\Lambda\qquad\text{in }\mathcal{V}^{\prime}.\label{ZoweEq-1}
\end{align}
Observe that $c_{\xi,1}\geq0$ by a standard maximum principle \cite[p. 327]{Evans}.
Now for 
\[
c_{y}=\zeta_{1}-\varepsilon_{\gamma}c_{\xi,1},
\]
we obtain 
\begin{align*}
c_{y}+\varepsilon_{\gamma}c_{\xi}-\zeta & =c_{y}+\varepsilon_{\gamma}c_{\xi,1}-\zeta_{1}+\varepsilon_{\gamma}c_{\xi,2}-\zeta_{2},\\
 & =\varepsilon_{\gamma}c_{\xi,2}-\zeta_{2},\\
 & =\max\left(0,\,z\right)-\max\left(0,\,-z\right)=z.
\end{align*}
and 
\begin{align*}
Ac_{y} & =A\zeta_{1}-\varepsilon_{\gamma}Ac_{\xi,1},\\
 & =c_{\xi,2}+z^{\prime}+\Lambda-\varepsilon_{\gamma}Ac_{\xi,1},\\
 & =c_{\xi,2}+z^{\prime}+c_{\xi,1},\\
 & =c_{\xi}+z^{\prime}.
\end{align*}
Here we have used \eqref{ZoweIn1-1}-\eqref{ZoweIn2-1} and \eqref{ZoweEq-1}.
Therefore, problem $\left(\mathcal{O}_{\gamma}^{n}\right)$ constraints
are qualified and the next proposition holds true. 
\begin{prop}
Let $\left(y_{\gamma}^{n},\,\xi_{\gamma}^{n},\,u_{\gamma}^{n}\right)$
be a solution for the problem $\left(\mathcal{O}_{\gamma}^{n}\right).$
Then there exists a Lagrange multiplier vector $\left(p_{\gamma}^{n},\,\lambda_{\gamma}^{n}\right)$
in $\mathcal{V}\times L^{2}\left(\Gamma_{C}\right)$ such that the
following first order optimality system holds 
\begin{align}
Ay_{\gamma}^{n}-Bu_{\gamma}^{n}-\xi_{\gamma}^{n}-\overline{y_{\gamma}}^{n-1}+\overline{\xi_{\gamma}}^{n-1} & =0\text{ in }\mathcal{V}^{\prime},\label{OS-gamma-eq-1}\\
Ap_{\gamma}^{n}+y_{\gamma}^{n}-y_{d}^{n}+\gamma\xi_{\gamma}^{n}-\lambda_{\gamma}^{n} & =0\text{ in }\mathcal{V}^{\prime},\label{OS-gamma-eq-2}\\
y_{\gamma}^{n}+\varepsilon_{\gamma}\xi_{\gamma}^{n}\geq0\qquad\lambda_{\gamma}^{n}\geq0,\qquad\left(y_{\gamma}^{n}+\varepsilon_{\gamma}\xi_{\gamma}^{n},\,\lambda_{\gamma}^{n}\right) & =0,\label{OS-gamma-eq-3}\\
\xi_{\gamma}^{n}\geq0,\qquad\left(\xi_{\gamma}^{n}-\xi_{d}^{n}+2\gamma\varepsilon_{\gamma}\xi_{\gamma}^{n}+\gamma y_{\gamma}^{n}-p_{\gamma}^{n}-\varepsilon_{\gamma}\lambda_{\gamma}^{n},\,\tau-\xi_{\gamma}^{n}\right) & \geq0,\label{OS-gamma-eq-4}\\
u_{\gamma}^{n}\geq0,\qquad\left(\alpha u_{\gamma}^{n}-\tau B^{*}p_{\gamma}^{n},\,v-u_{\gamma}^{n}\right)_{\Gamma_{C}} & \geq0,\label{OS-gamma-eq-5}
\end{align}
where $\tau$ and $v$ are two non-negative arbitrary functions in
$L^{2}\left(\Omega\right)$ and $L^{2}\left(\Gamma_{C}\right)$ respectively.
$B^{*}$ is the adjoint operator of $B.$ 
\end{prop}

\begin{rem}
Conditions \eqref{OS-gamma-eq-4} and \eqref{OS-gamma-eq-4} correspond
to the projection of $\dfrac{1}{1+2\gamma\varepsilon_{\gamma}}\left(\xi_{d}^{n}+\lambda_{\gamma}^{n}\right)$
and $\dfrac{\tau}{\alpha}B^{*}p_{\gamma}^{n}$ over the non-negative
cones in $L^{2}\left(\Omega\right)$ and $L^{2}\left(\Gamma_{C}\right)$
respectively: 
\[
u_{\gamma}^{n}=\max\left(0,\,\dfrac{\tau}{\alpha}B^{*}p_{\gamma}^{n}\right),\qquad\xi_{\gamma}^{n}=\max\left(0,\,\dfrac{1}{1+2\gamma\varepsilon_{\gamma}}\left(p_{\gamma}^{n}+\xi_{d}^{n}+\varepsilon_{\gamma}\lambda_{\gamma}^{n}-\gamma y_{\gamma}^{n}\right)\right).
\]
\end{rem}

\section{Numerical experiments}

In this section we present two preliminary numerical experiments to
assess the validity of the above developed theoretical procedure.
At each time step $t^{n}=n\tau$ we solve a discrete version of the
optimization problem $\left(\mathcal{O}^{n}\right)_{\gamma_{k}}$
for a sequence of penalty parameters $\left(\gamma_{k}\right)_{k\in\mathbb{N}}$
with $\gamma_{k}=10^{-3}\times1.5^{k}$ and $k=1,\dots40.$ We select
a regularization parameter $\varepsilon_{\gamma_{k}}=\dfrac{1}{10^{3}+\gamma_{k}^{4}}.$
The parameter for the cost of the control is taken $\nu=10^{-4}.$
All functions are discretized by continuous piecewise linear finite
elements. The fully discretized penalized-regularized control problems
corresponding to $\left(\mathcal{O}^{n}\right)_{\gamma_{k}}$ are
then solved numerically using the fmincon Matlab function.

\subsection*{Example 1}

We consider a one dimensional free convection problem with a known
analytical solution \cite{Esen}: 
\[
y_{ex}\left(x,t\right)=\begin{cases}
\exp\left(t-x\right)-1 & \text{ if }0\leq x\leq t,\\
0 & \text{ if }t\leq x\leq x_{max},
\end{cases}\;\;\xi_{ex}\left(x,t\right)=\begin{cases}
0 & \text{ if }0\leq x\leq t,\\
1 & \text{ if }t\leq x\leq x_{max}.
\end{cases}
\]
Our aim here is to apply a heat flux on the left boundary, $\Gamma_{C}=\{0\},$
to get temperature and solid fraction profiles as close as possible
to the exact solution. For the instantaneous boundary control problem
we choose a fixed time step $\tau=0.01$ and we set 
\[
y_{d}^{n}=y_{ex}\left(x,\,n\tau\right),\qquad\xi_{d}^{n}=\xi_{ex}\left(x,\,n\tau\right)\qquad\qquad\text{ for }n=1,\dots N=300.
\]
For the computational domain we choose a uniform grid of size $h=0.01$
with $x_{max}=4.$ The analytical control $u_{ex}\left(t\right)=\exp(t)$
is very well reconstructed up to the first few iterations as shown
in Fig.~\ref{fig-ex1-control}. An excellent agreement has been found
between the analytical and controlled temperature profiles as depicted
in Fig.~\ref{fig-ex1-y}. The complementarity condition between the
temperature and solid fraction are respected, as shown in Fig.~\ref{fig-ex1-comp}
for the sample instant $t=3$, which emphasize the relevance of the
developed regularization-penalization approach.

\begin{figure}
\begin{centering}
\includegraphics[width=12cm,height=4cm]{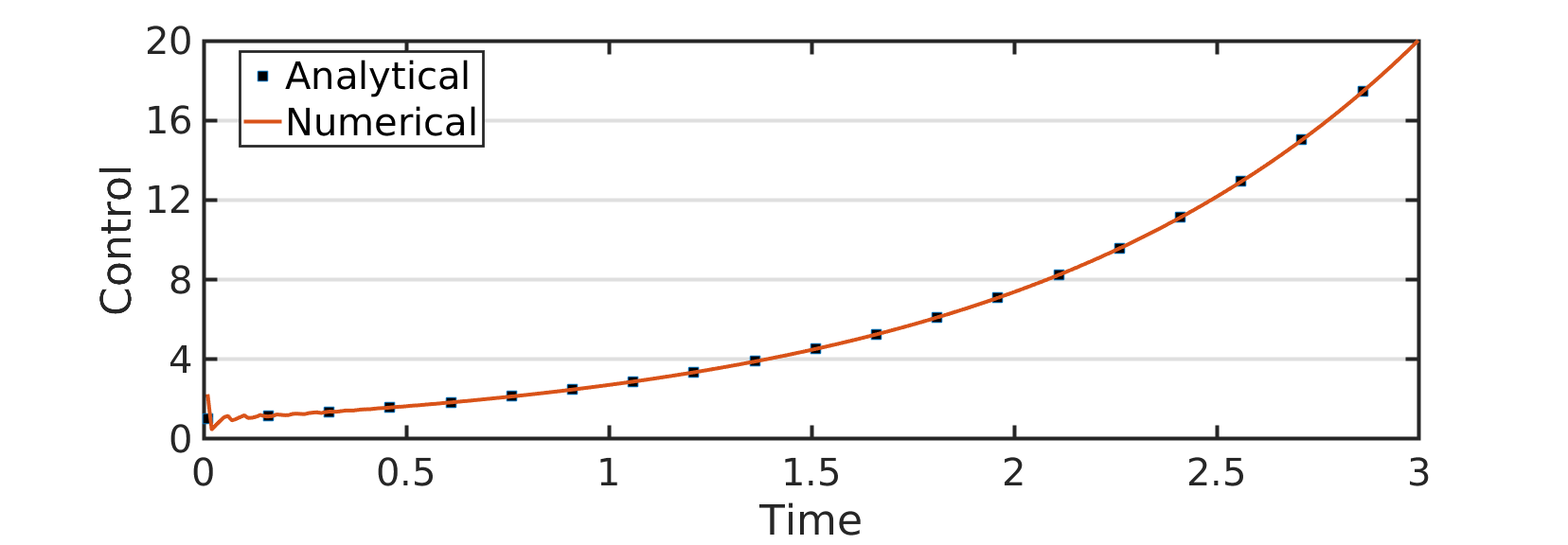} 
\par\end{centering}
\caption{\label{fig-ex1-control}Analytical and computed sub-optimal controls}
\end{figure}

\begin{figure}
\begin{centering}
\includegraphics[width=12cm,height=6cm]{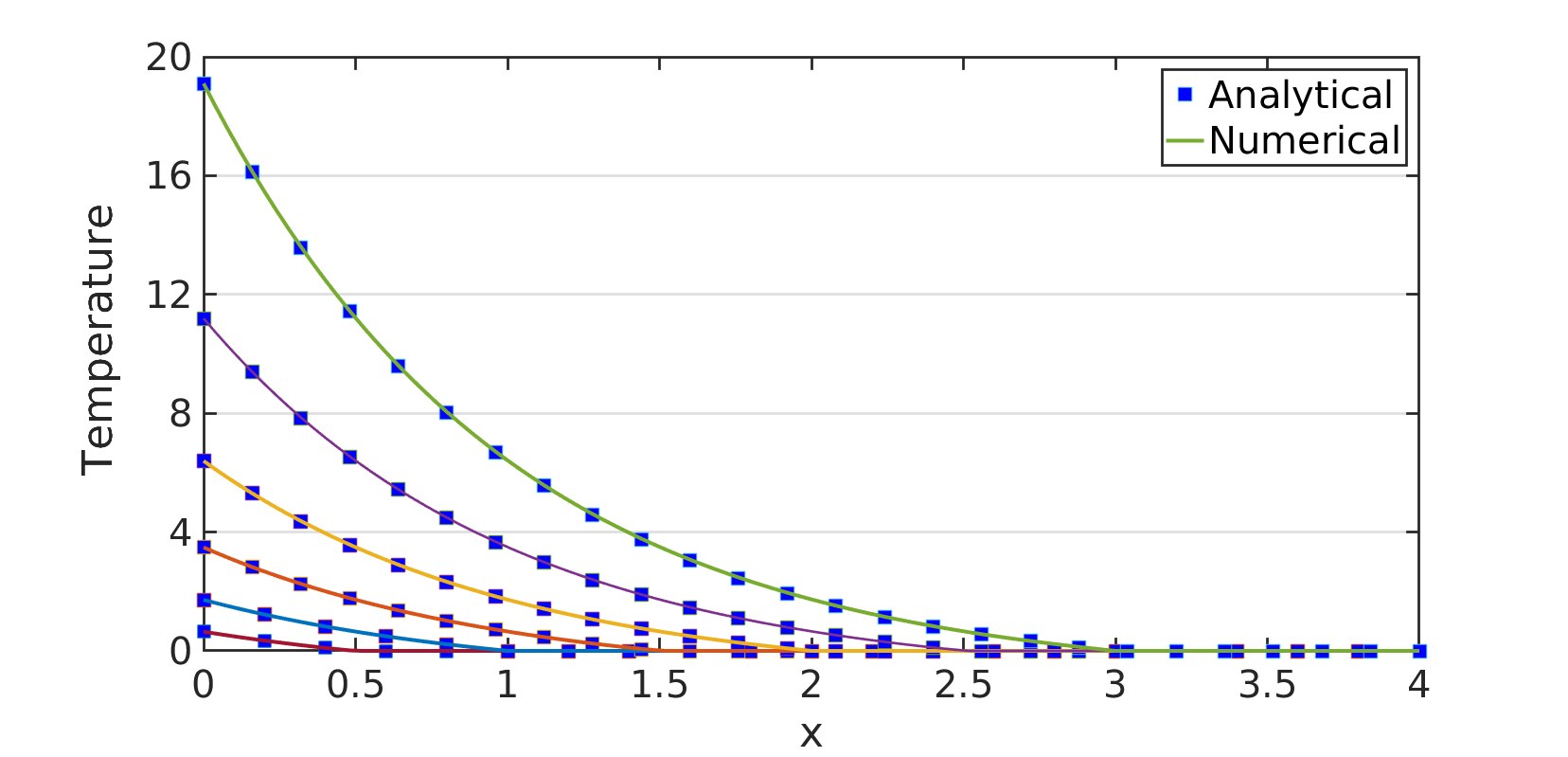} 
\par\end{centering}
\caption{\label{fig-ex1-y}Analytical and computed temperature profiles at
different instants}
\end{figure}

\begin{figure}
\begin{centering}
\includegraphics[width=12cm,height=6cm]{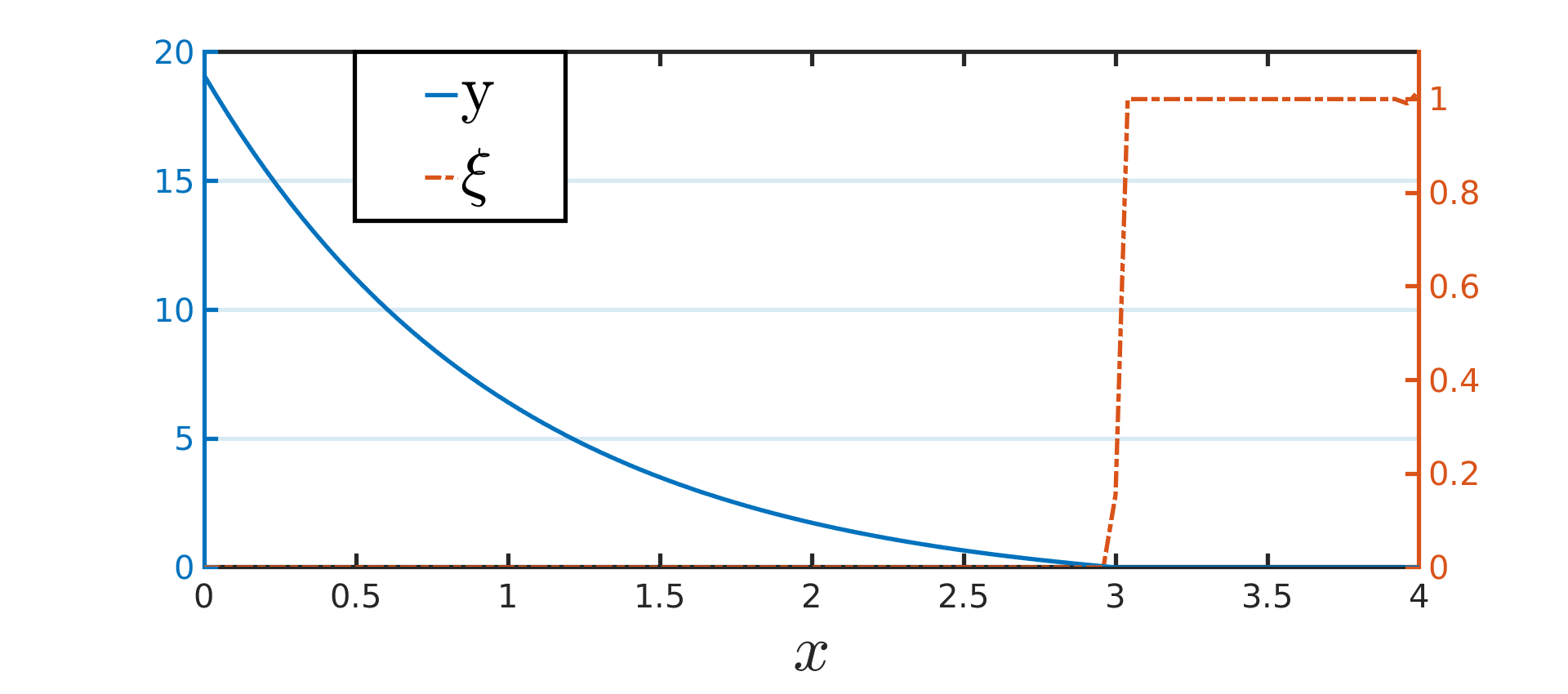} 
\par\end{centering}
\caption{\label{fig-ex1-comp}Computed temperature and solid fraction profiles
at time $t=t_{f}=3$}
\end{figure}

\subsection*{Example 2}

Here we consider a two dimensional problem where the computational
domain $\Omega=\left]0,\,2\right[\times\left]0,\,4\right[$ is discretized
using a $50\times100$ uniform grid. The time step is taken $\tau=0.1$
and a constant convection velocity $\overrightarrow{\bm{v}}=\left(\begin{array}{c}
-0.5\\
0
\end{array}\right)$ is used. No-heat flux condition is applied on the right boundary
and a temperature $y=0$ is held at the top and bottom. We apply the
heat flux control on $\Gamma_{C}:=\left\{ x\in\Omega\,:\,x_{1}=0\right\} $
to govern the system toward the following time-independent desired
state 
\[
y_{d}^{n}\left(x\right)=y_{d}\left(x\right)=\begin{cases}
\exp\left(\frac{1}{4}\left(4-x_{2}\right)x_{2}-x_{1}\right)-1 & \text{ if }x_{1}\leq\frac{1}{4}\left(4-x_{2}\right)x_{2},\\
0 & \text{ if }x_{1}\geq\frac{1}{4}\left(4-x_{2}\right)x_{2},
\end{cases}
\]
\[
\xi_{d}^{n}\left(x\right)=\xi_{d}\left(x\right)=\begin{cases}
0 & \text{ if }x_{1}<\frac{1}{4}\left(4-x_{2}\right)x_{2},\\
1 & \text{ if }x_{1}\geq\frac{1}{4}\left(4-x_{2}\right)x_{2}.
\end{cases}
\]
Figs.~\ref{fig-ex2-y}-\ref{fig-ex2-xi} show the evolution of temperature
$y$ and solid fraction $\xi$ driven by the sub-optimal controls
towards the desired state. A fairly good approximation is obtained.
The significant reduction of the cost functional value is achieved
during the first five time steps and almost stagnates up to $t\approx1$
as shown in Fig.~\ref{fig-ex2-J}. In Fig.~\ref{fig-ex2-u} we present
the computed sub-optimal control at sample instances. We observe a
strong control at the first time step getting inactive near the boundaries.
The controls shape is consistent with the desired state one.

\begin{figure}
\begin{centering}
\includegraphics[width=4cm,height=6cm]{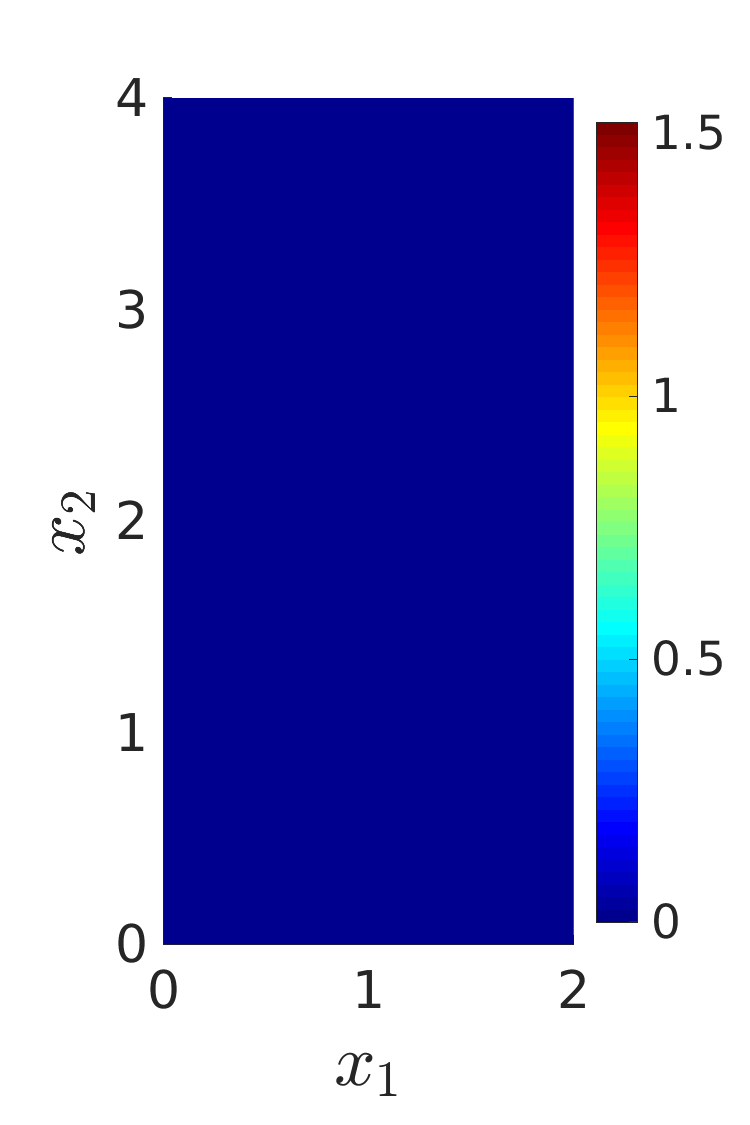}\includegraphics[width=4cm,height=6cm]{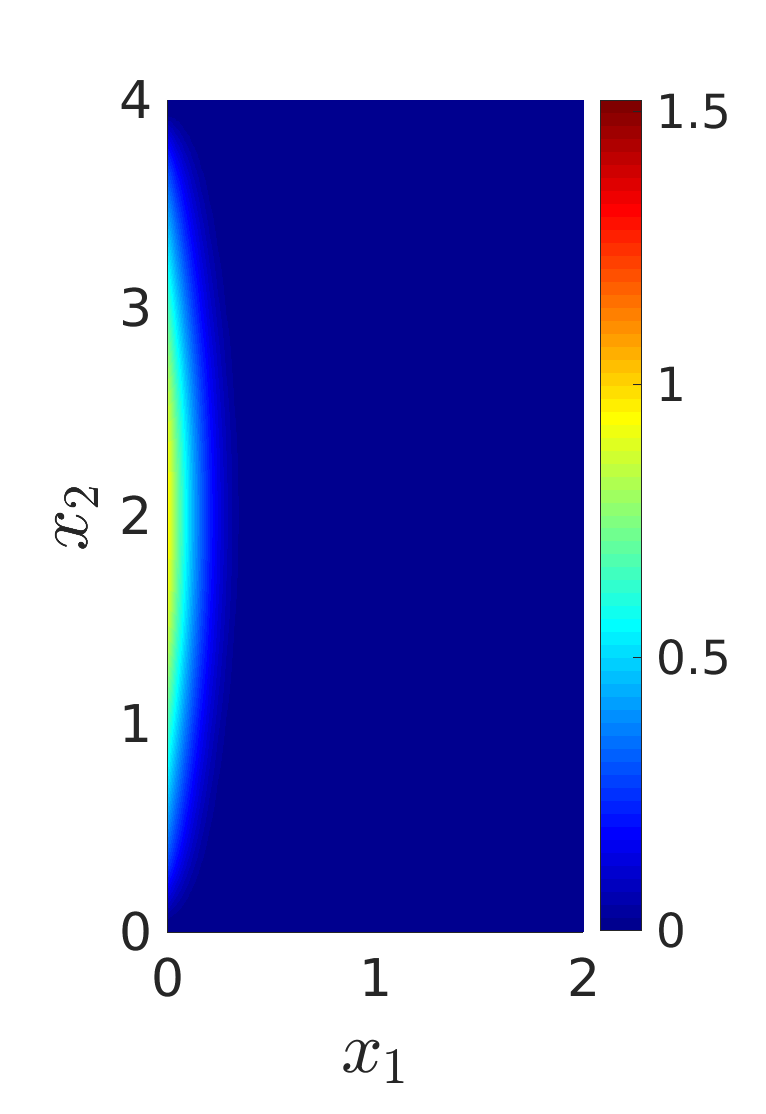}\includegraphics[width=4cm,height=6cm]{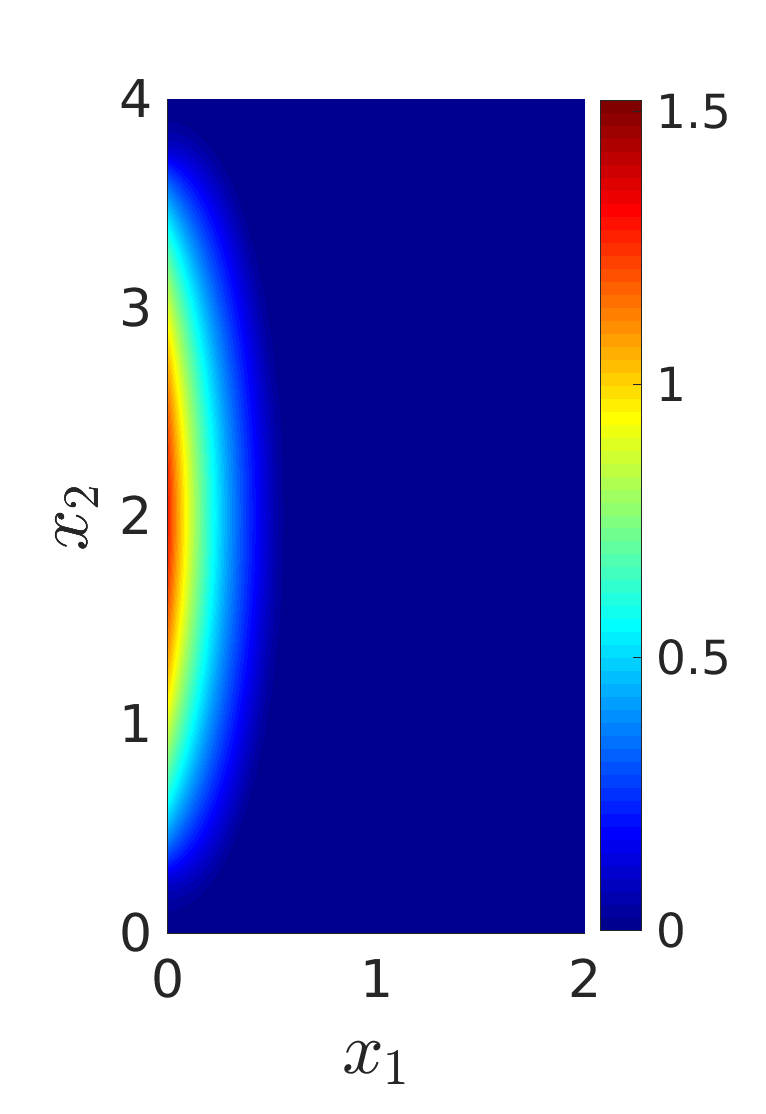}\\
 \includegraphics[width=4cm,height=6cm]{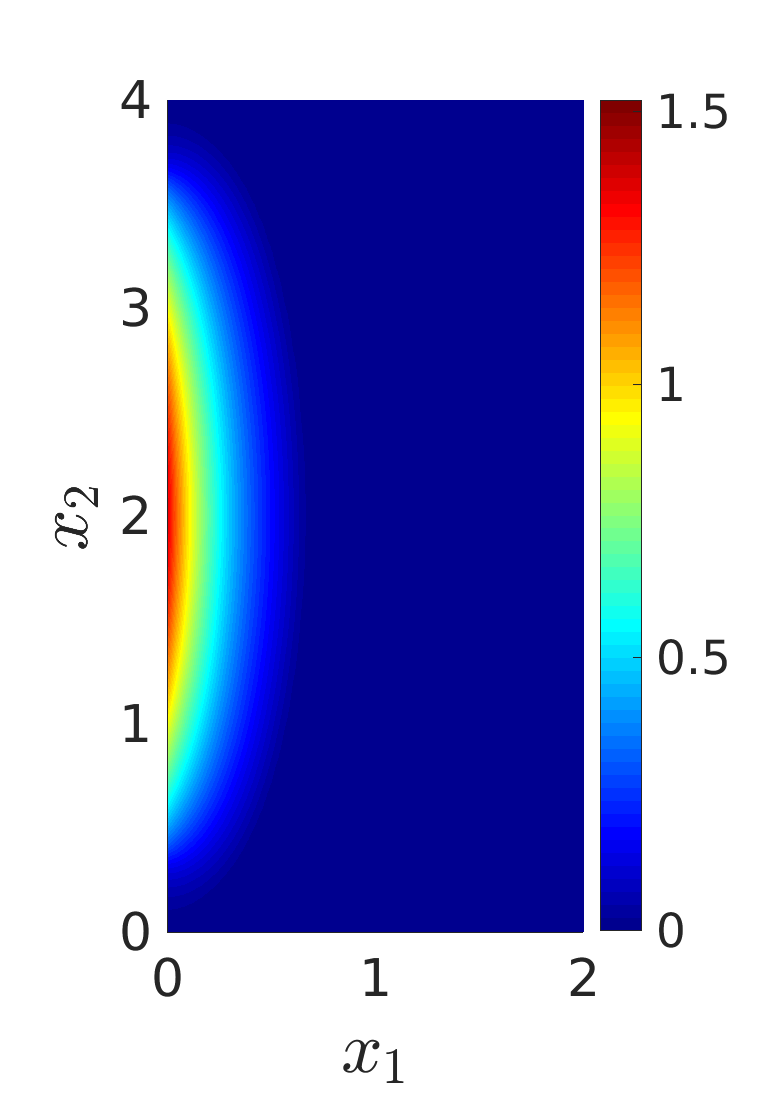}\includegraphics[width=4cm,height=6cm]{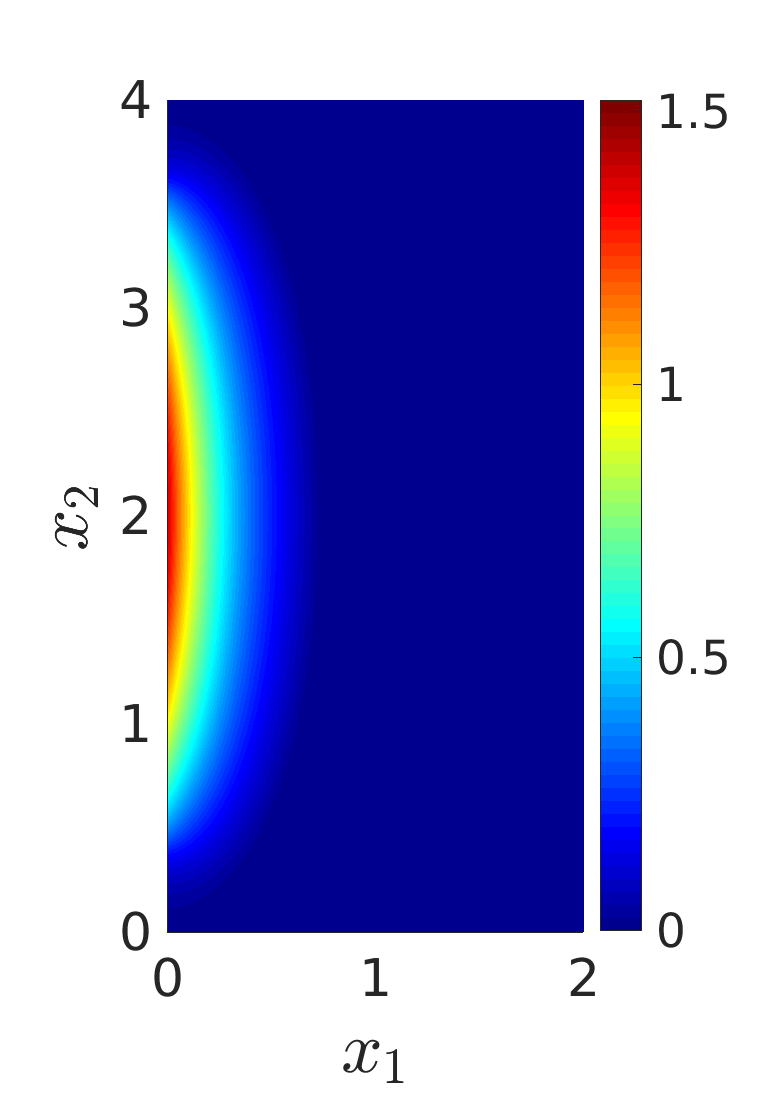}\includegraphics[width=4cm,height=6cm]{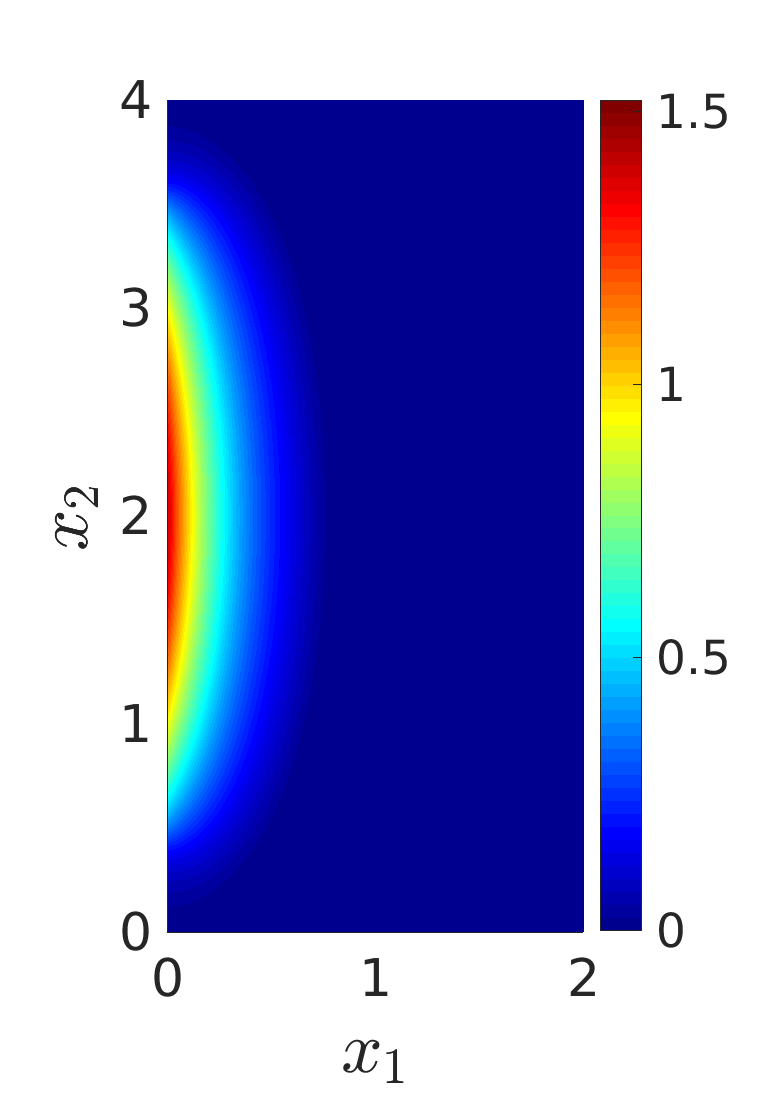}\\
 \includegraphics[width=4cm,height=6cm]{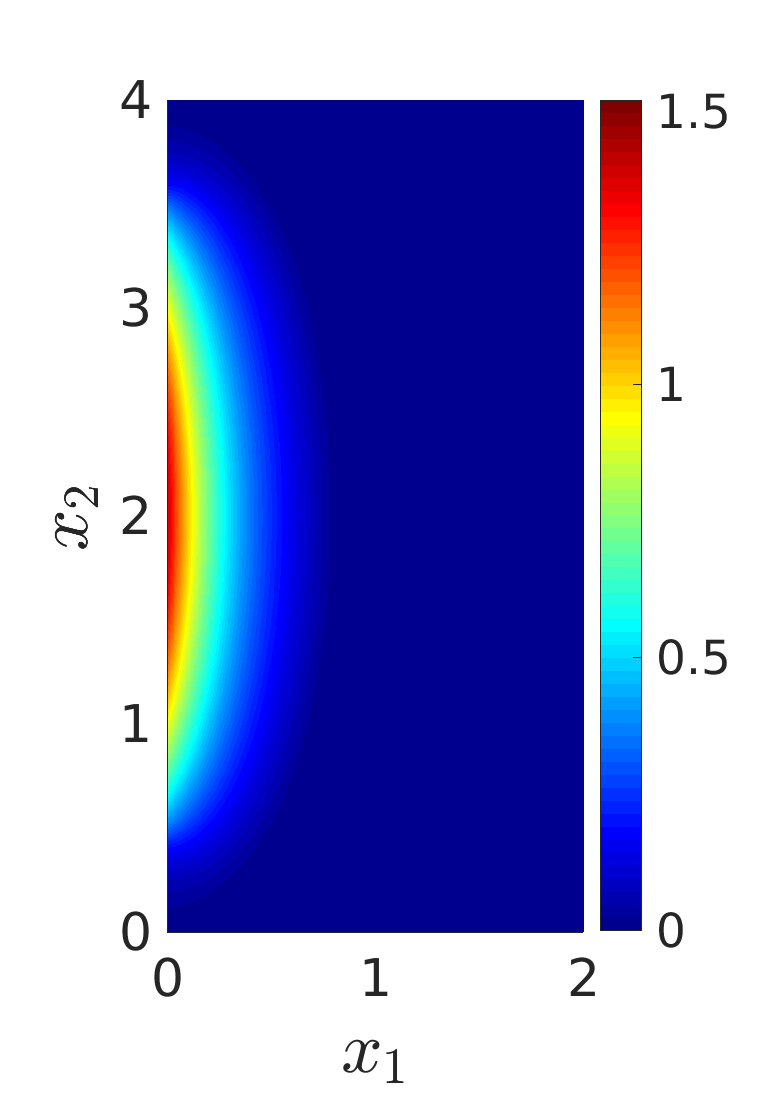}\includegraphics[width=4cm,height=6cm]{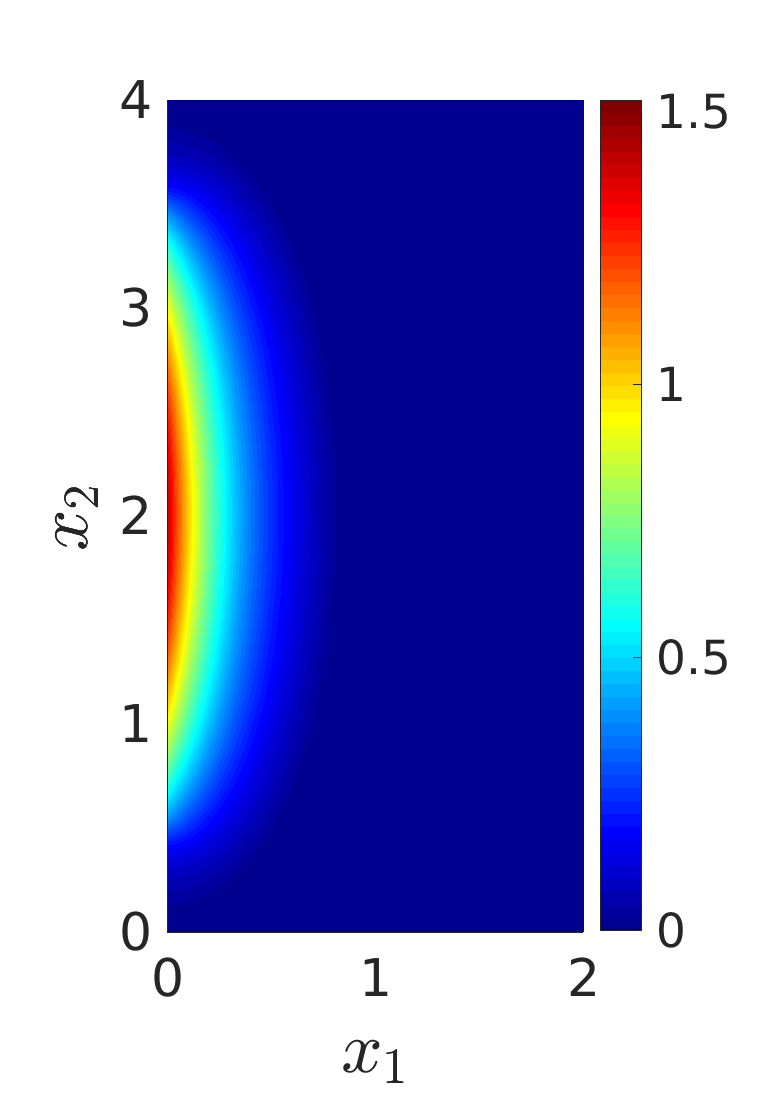}\includegraphics[width=4cm,height=6cm]{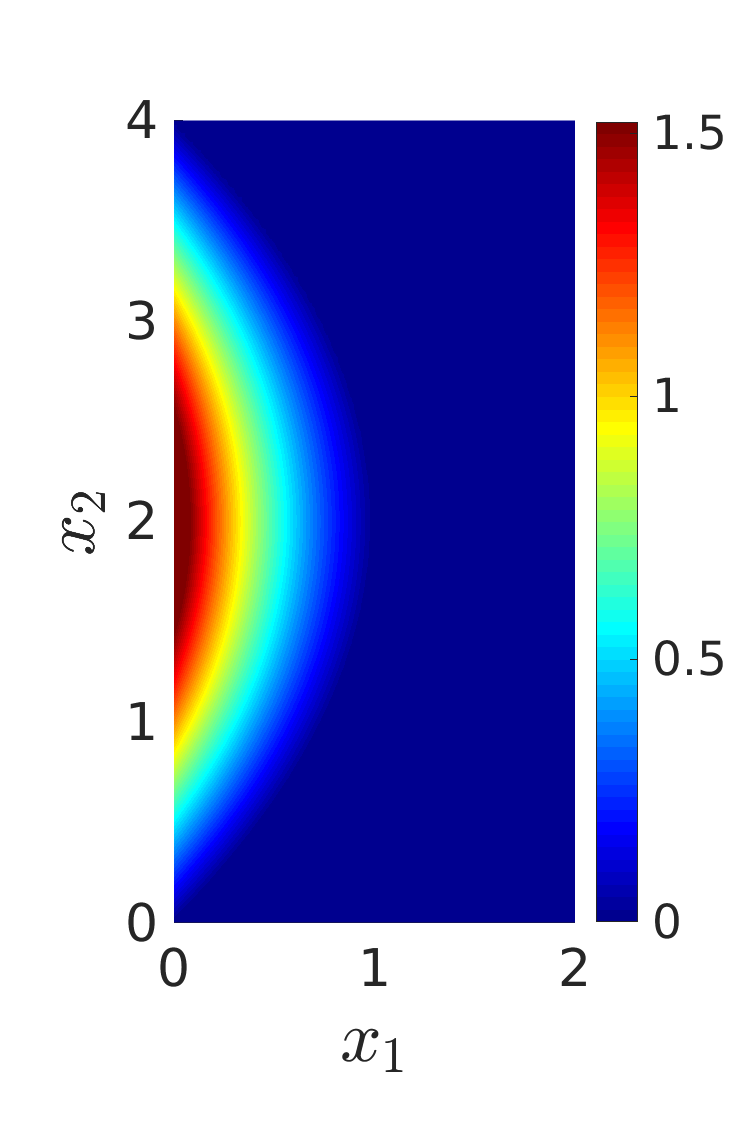} 
\par\end{centering}
\caption{\label{fig-ex2-y}Computed temperatures at $t=0,\,0.1,\,0.3,\,0.5,\,0.7,\,0.9,\,1.1,\,1.4$
and the desired temperature profile}
\end{figure}

\begin{figure}
\begin{centering}
\includegraphics[width=4cm,height=6cm]{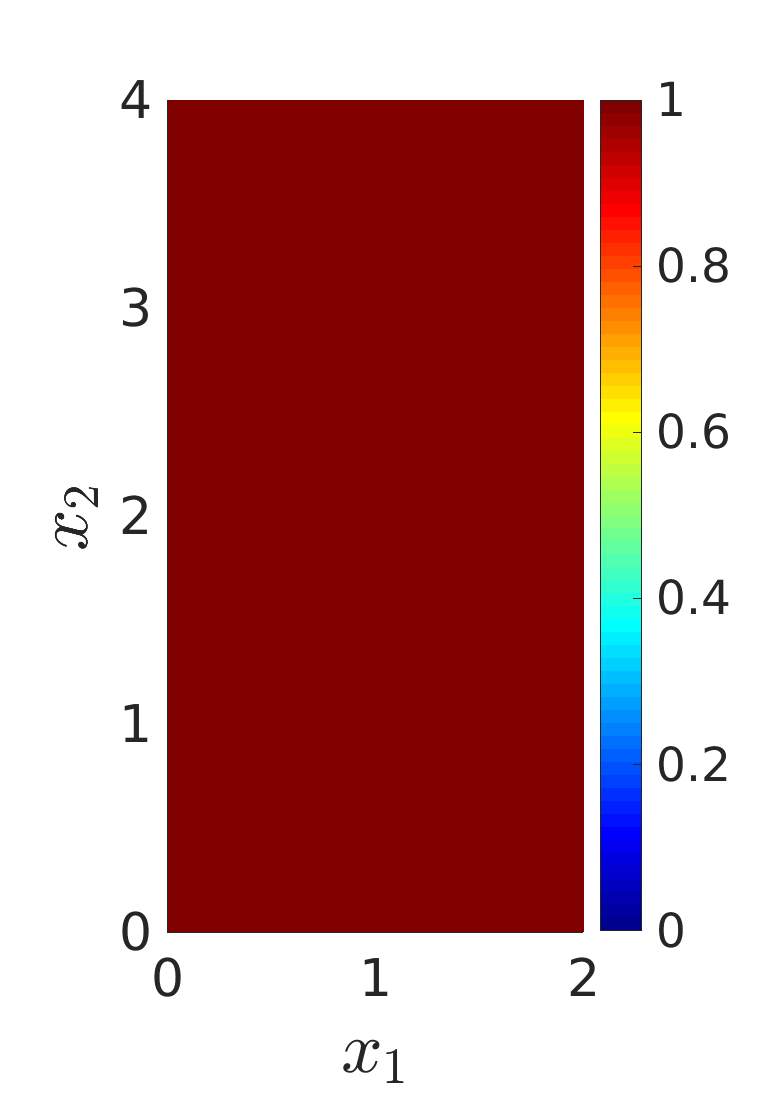}\includegraphics[width=4cm,height=6cm]{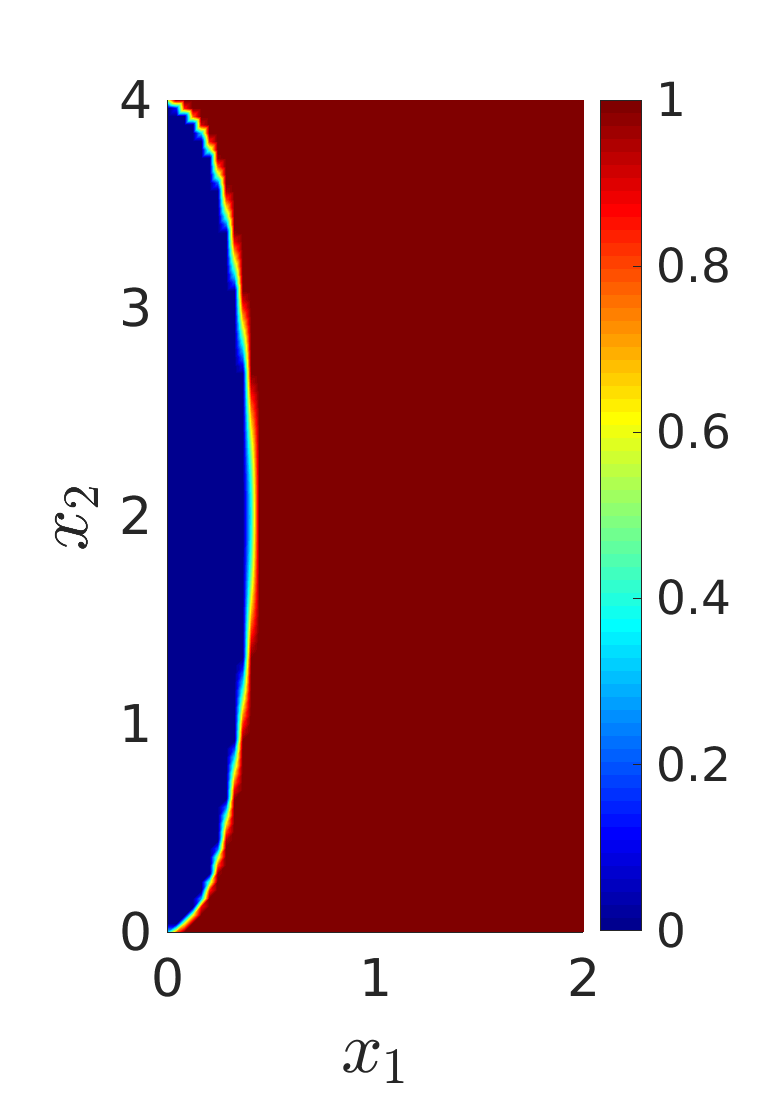}\includegraphics[width=4cm,height=6cm]{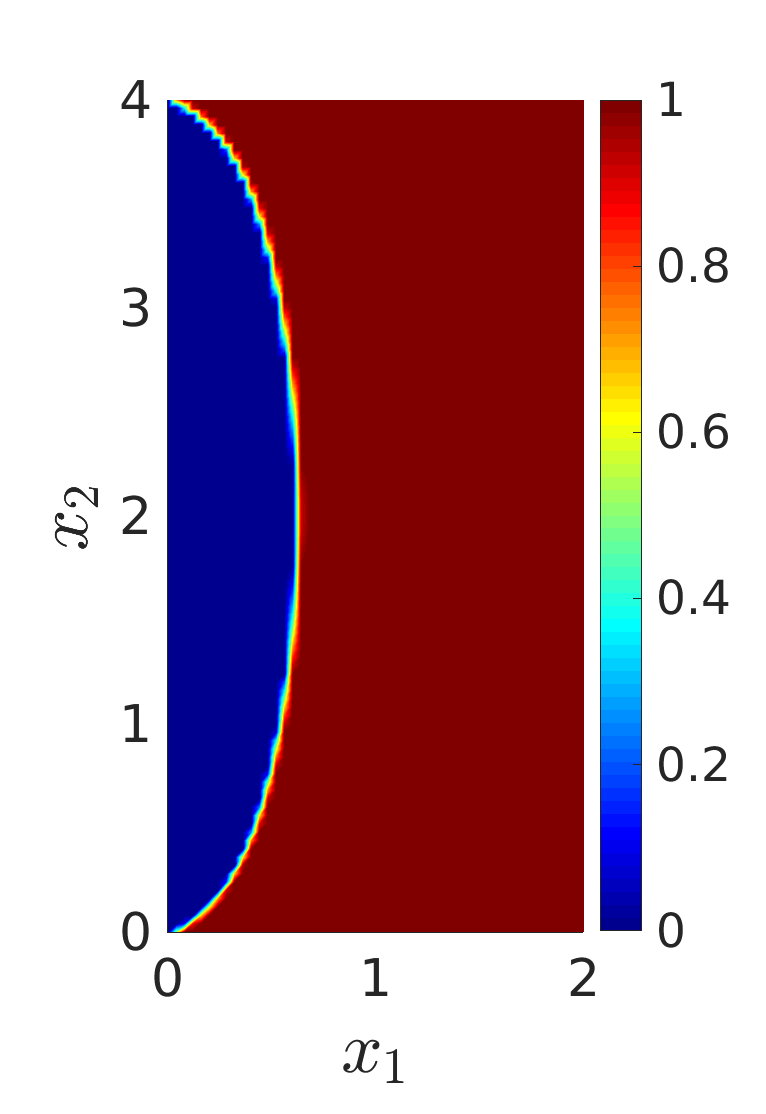}\\
 \includegraphics[width=4cm,height=6cm]{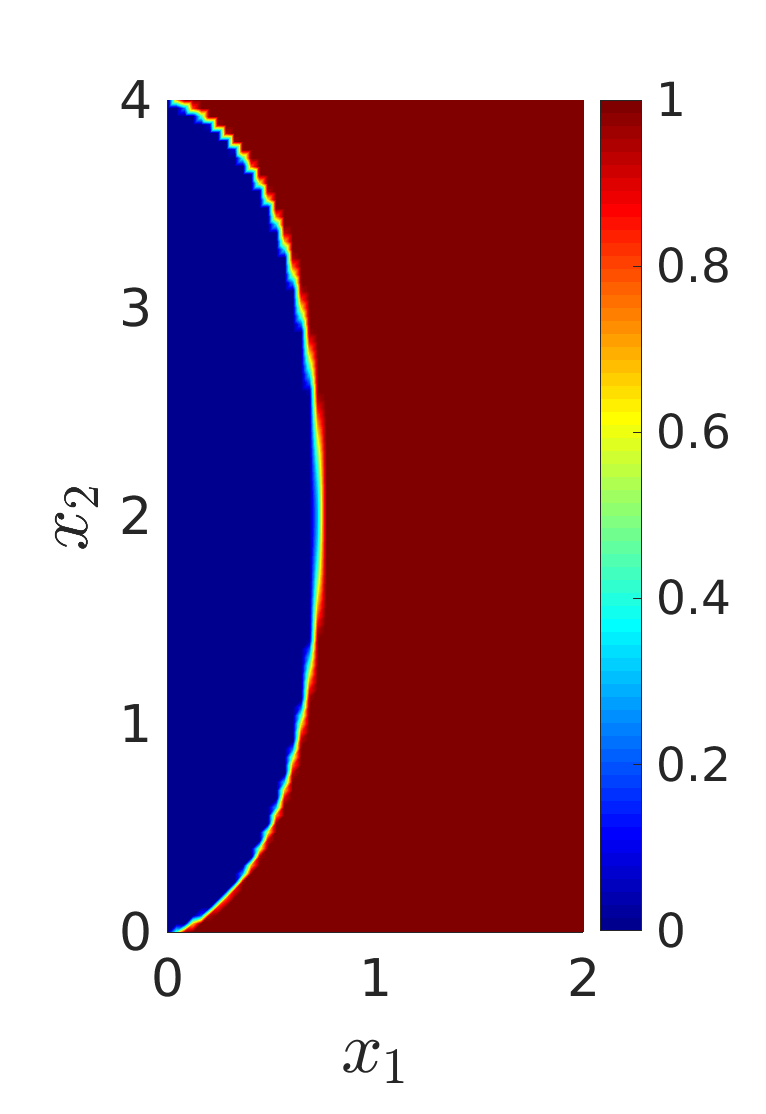}\includegraphics[width=4cm,height=6cm]{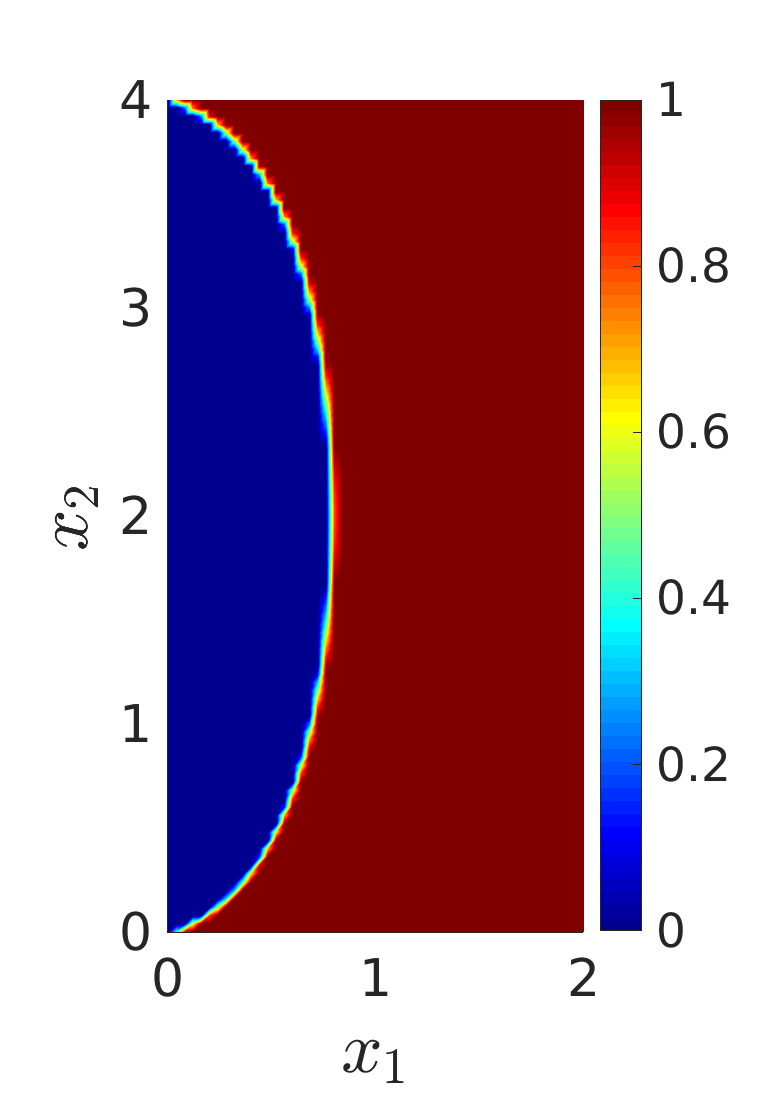}\includegraphics[width=4cm,height=6cm]{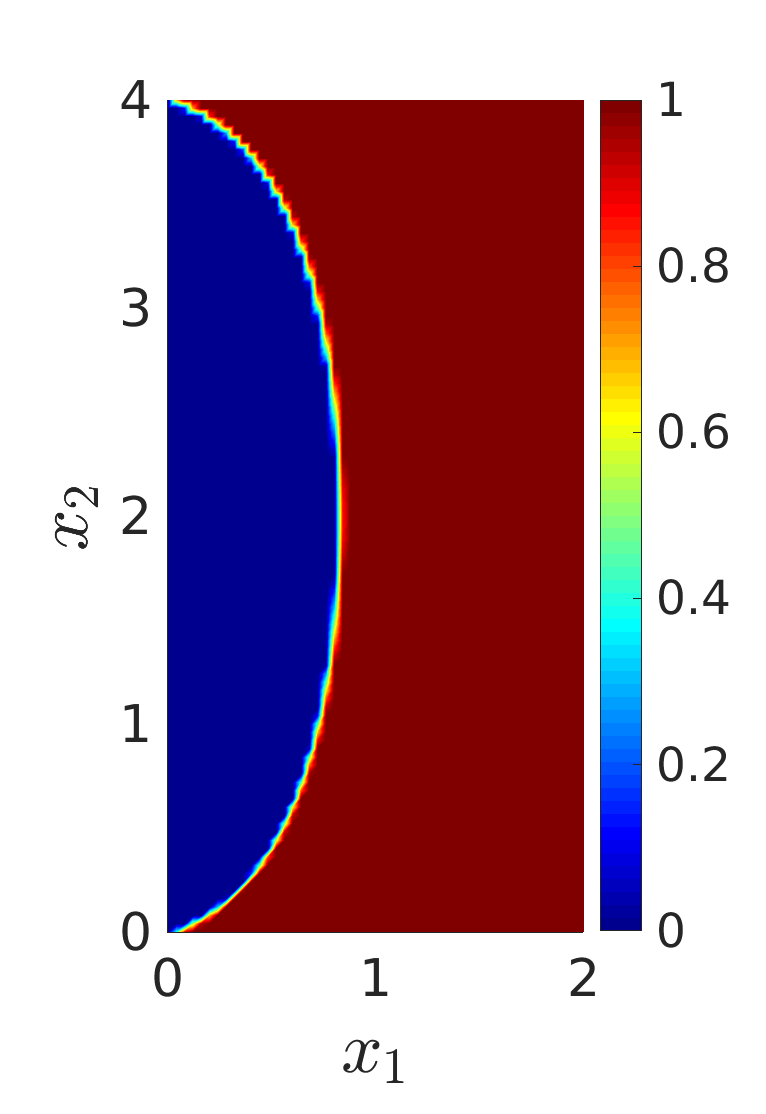}\\
 \includegraphics[width=4cm,height=6cm]{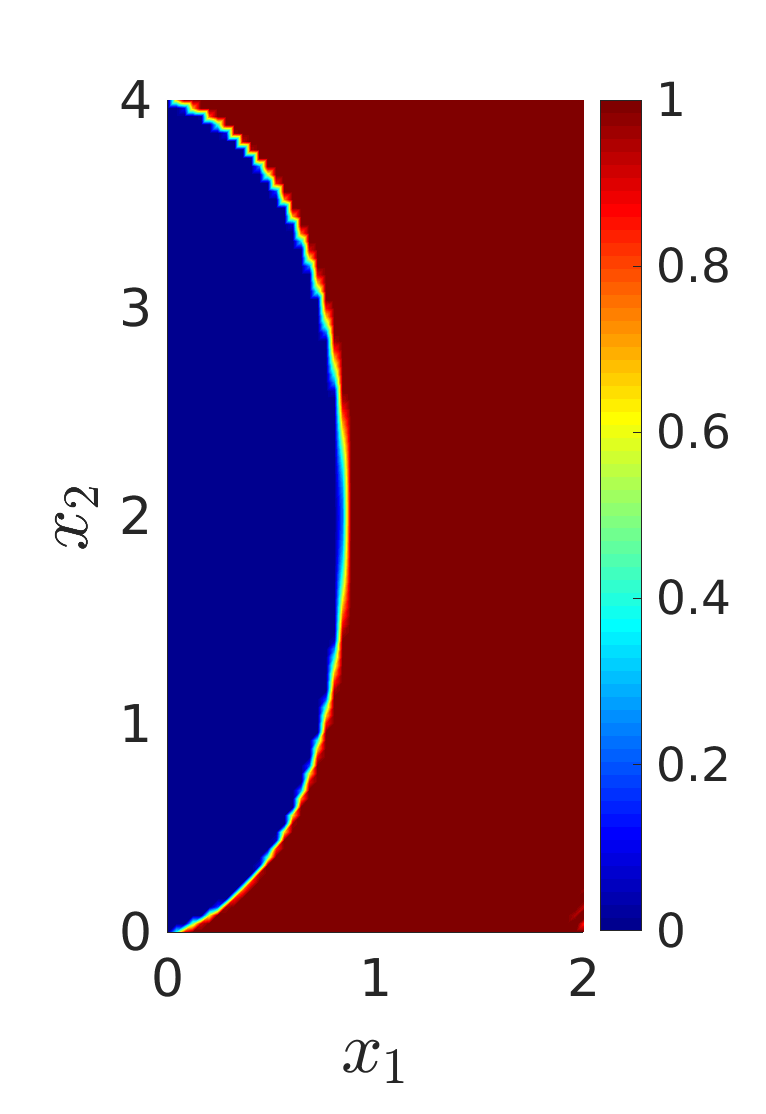}\includegraphics[width=4cm,height=6cm]{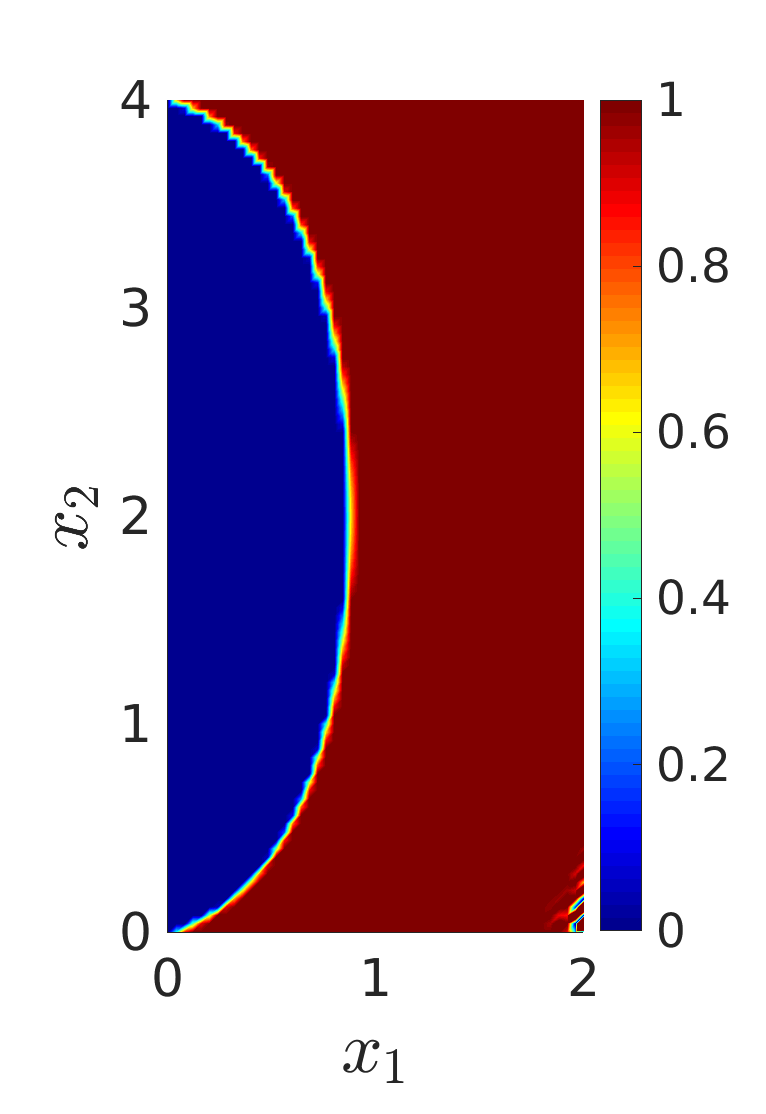}\includegraphics[width=4cm,height=6cm]{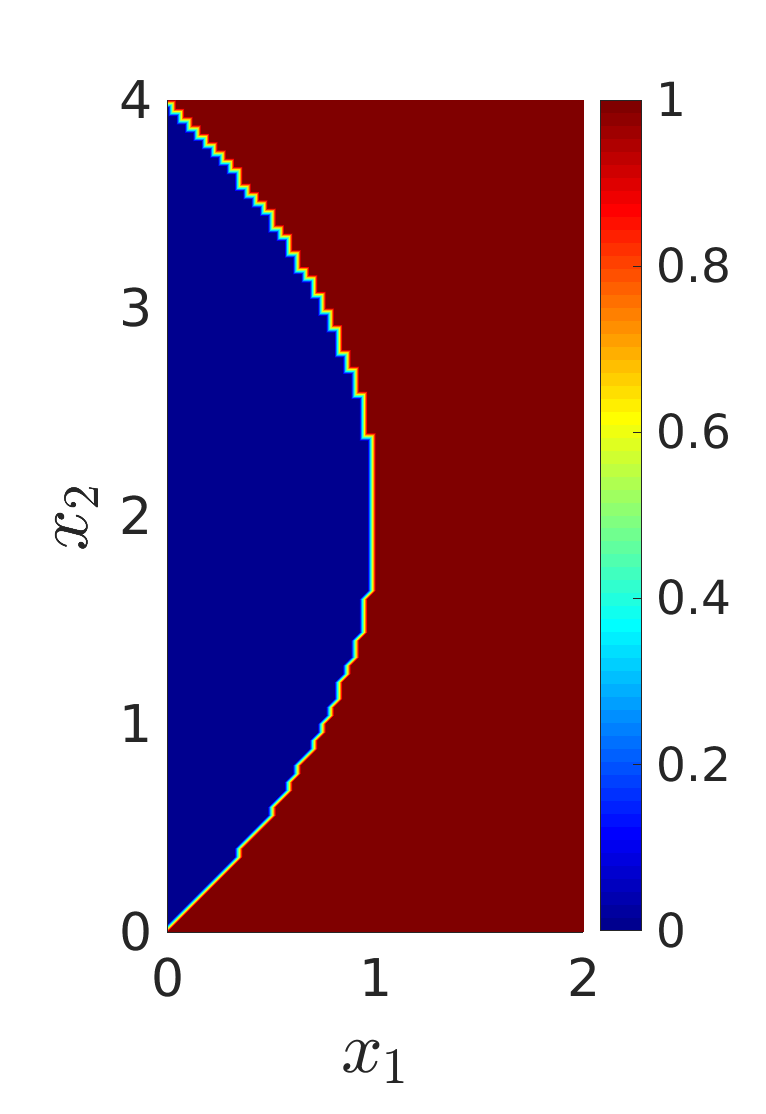} 
\par\end{centering}
\caption{\label{fig-ex2-xi}Computed solid fraction at $t=0,\,0.1,\,0.3,\,0.5,\,0.7,\,0.9,\,1.1,\,1.4$
and the desired solid fraction profile}
\end{figure}

\begin{figure}
\begin{centering}
\includegraphics[width=8cm,height=5cm]{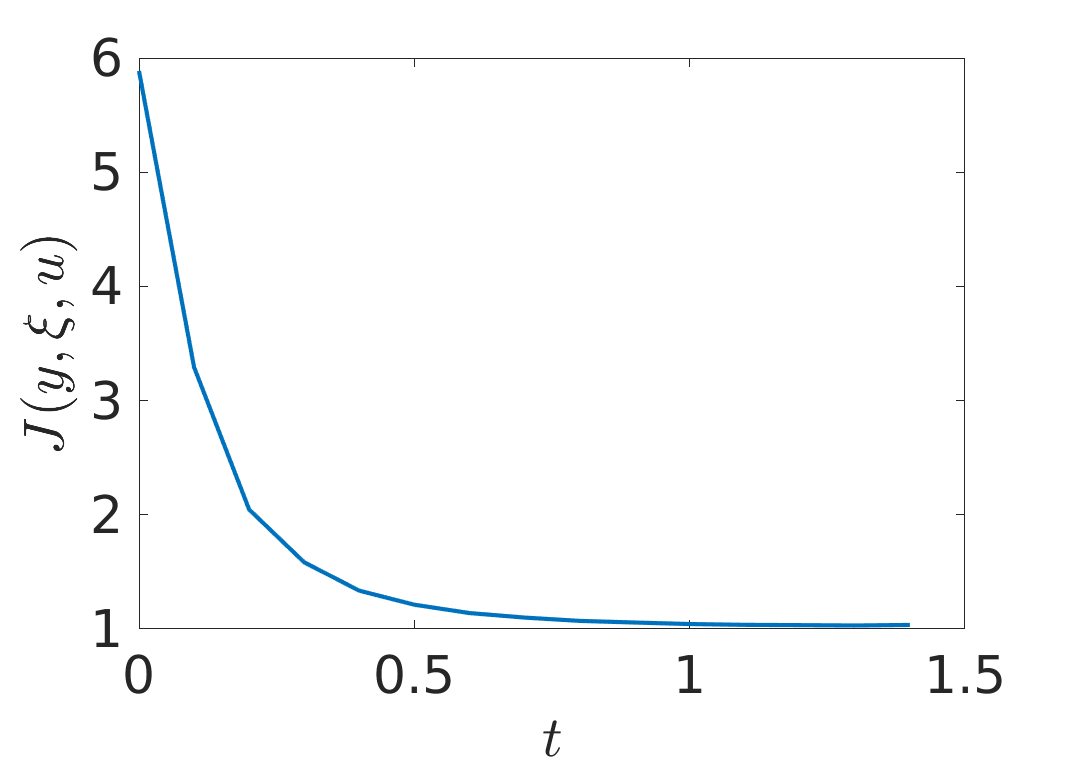} 
\par\end{centering}
\caption{\label{fig-ex2-J}Reduction of the cost functional}
\end{figure}

\begin{figure}
\begin{centering}
\includegraphics[width=10cm,height=5cm]{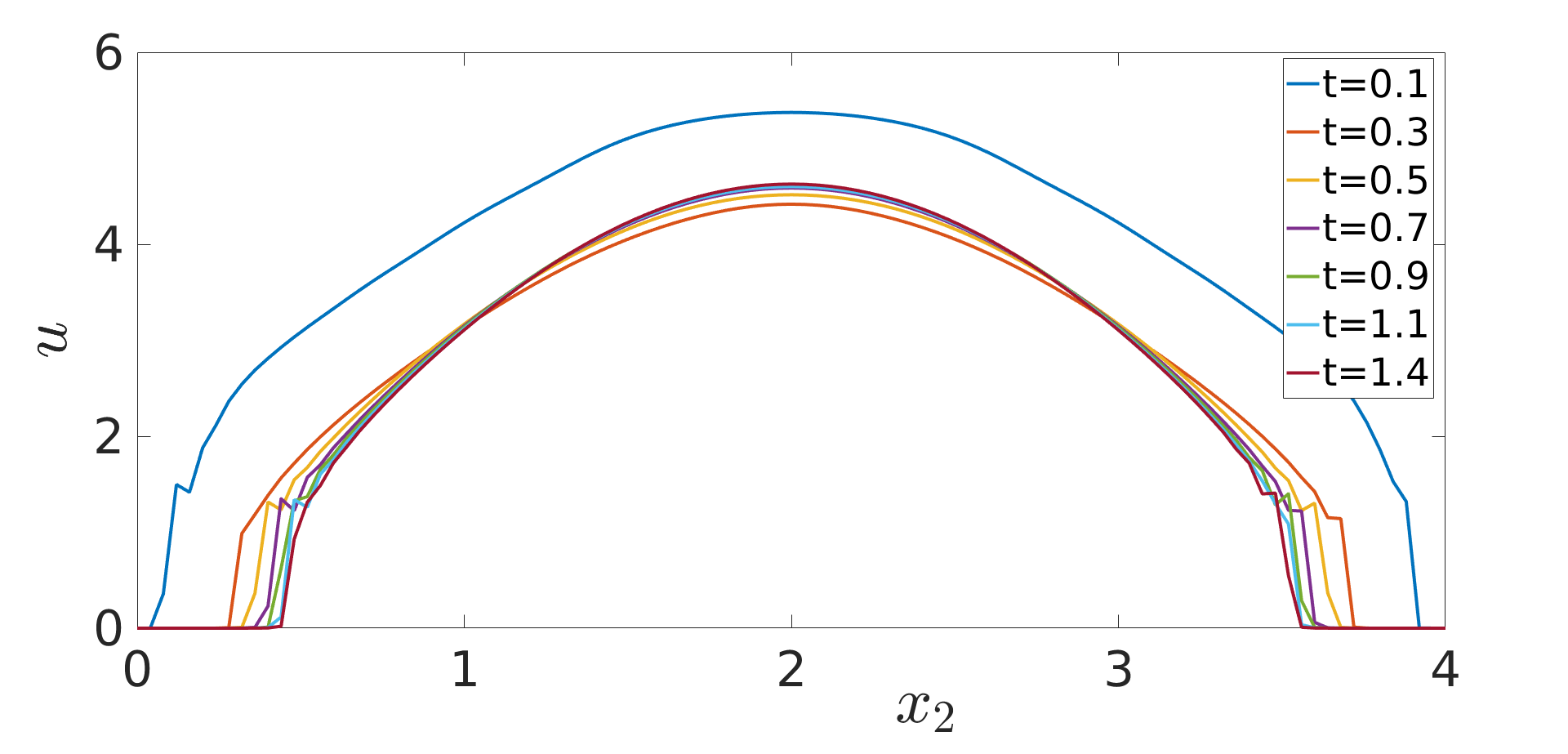} 
\par\end{centering}
\caption{\label{fig-ex2-u}Computed sub-optimal controls at different time
steps}
\end{figure}

\begin{rem}
To apply the developed approach on more realistic benchmarks, a coupling
with momentum and mass conservation equations is required. However,
many discretization and algorithmic aspects have to be developed first.
Questions related to adaptive mesh refinement, selection of the optimization
parameters, solution algorithm and preconditioning will be addressed
in a forthcoming study. 
\end{rem}

\section*{Appendix A: proof of Theorem \ref{thm-1}}

To show that the problem 
\[
\tag{\ensuremath{\mathcal{WF}^{n}}}\left\{ \begin{array}{ll}
Ay^{n}=\xi^{n}+Bu^{n}+\overline{y}^{n-1}-\overline{\xi}^{n-1} & \text{in }\mathcal{V}^{\prime},\\
\xi^{n}\in\mathcal{H}(y^{n}), & \text{a.e. in }\Omega
\end{array}\right.
\]
has a solution, let $\mathcal{H}_{\varepsilon}$ be a regularization
of the Heaviside operator $\mathcal{H}$ given by 
\[
\mathcal{H}_{\varepsilon}(x)=\left\{ \begin{array}{ll}
0 & \quad\textrm{ if }x\geq\varepsilon,\\
1-x/\varepsilon & \quad\textrm{ if }0\leq x\leq\varepsilon,\\
1 & \quad\textrm{ if }x\leq0.
\end{array}\right.
\]

Correspondingly, we consider the following regularized problem 
\[
\tag{\ensuremath{\mathcal{WF}_{\varepsilon}^{n}}}\left\{ \begin{array}{rcl}
\text{Find }y_{\varepsilon}^{n}\in\mathcal{V}\text{ such that }\\
y_{\varepsilon}^{n}\geq0 & \text{a.e in }\Omega,\\
Ay_{\varepsilon}^{n}=\mathcal{H}_{\varepsilon}(y_{\varepsilon}^{n})+Bu^{n}+\overline{y}^{n-1}-\overline{\xi}^{n-1}\qquad & \text{in }\mathcal{V}^{\prime}.
\end{array}\right.
\]

\begin{lem}
The regularized problem $\left(\mathcal{WF}_{\varepsilon}^{n}\right)$
has a unique solution $y_{\varepsilon}^{n}.$ Moreover there exists
a constant $C$ not depending on $\varepsilon$ such that 
\begin{equation}
\|y_{\varepsilon}^{n}\|_{H^{1}\left(\Omega\right)}\leq C.\label{eq-Reg-Sol-Est}
\end{equation}
\end{lem}

\begin{proof}
Consider the mapping $T$ which, for any $y_{\varepsilon}^{n}\in L^{2}\left(\Omega\right),$
associates $\widetilde{y_{\varepsilon}^{n}}=T(y_{\varepsilon}^{n})$
the solution of the following elliptic problem 
\begin{equation}
\begin{array}{rcl}
A\widetilde{y_{\varepsilon}^{n}}=\mathcal{H}_{\varepsilon}(y_{\varepsilon}^{n})+Bu^{n}+\overline{y}^{n-1}-\overline{\xi}^{n-1} & \qquad\text{in }\mathcal{V}^{\prime}.\end{array}\label{LinearizedRegularizedProblem}
\end{equation}
The problem $\eqref{LinearizedRegularizedProblem}$ has a unique solution
by Lax-Milgram theorem. Moreover, there exists a constant $Cst$ not
depending on $\varepsilon$ such that 
\begin{equation}
\|\widetilde{y_{\varepsilon}^{n}}\|_{H^{1}\left(\Omega\right)}\leq Cst.\label{Estimate1-1}
\end{equation}
Here, we have used the fact that $\left(\mathcal{H}_{\varepsilon}(y_{\varepsilon}^{n})\right){}_{\varepsilon}$
is bounded in $L^{\infty}\left(\Omega\right)$ independently of $\varepsilon.$

The mapping $T$ is then bounded from $L^{2}\left(\Omega\right)$
to $H^{1}\left(\Omega\right).$ From the compact embedding of $H^{1}\left(\Omega\right)$
into $L^{2}\left(\Omega\right)$, it follows that $T$ is completely
continuous from $\mathcal{V}$ to $L^{2}\left(\Omega\right).$ Moreover
the estimate $\eqref{Estimate1-1}$ shows that $T(B_{Cst})\subset B_{Cst}$
with $B_{Cst}$ being the $H^{1}\left(\Omega\right)$-ball of radius
$Cst.$ Schauder's fixed point theorem yields the existence of a function
$y_{\varepsilon}^{n}$ such that $T(y_{\varepsilon}^{n})=y_{\varepsilon}^{n}$
satisfying \eqref{eq-Reg-Sol-Est} with $C=Cst.$

Next, we claim that $y_{\varepsilon}^{n}\geq0$ a.e. in $\Omega.$
Let $\left(y_{\varepsilon}^{n}\right)^{-}=\min(0,\,y_{\varepsilon}^{n}).$
It is clear that $\left(y_{\varepsilon}^{n}\right)^{-}\in\mathcal{V}.$
By choosing $\phi=\left(y_{\varepsilon}^{n}\right)^{-}$ in \ref{LinearizedRegularizedProblem}
we arrive at 
\[
\left\langle Ay_{\varepsilon}^{n},\,\left(y_{\varepsilon}^{n}\right)^{-}\right\rangle =\left\langle Bu^{n},\,\left(y_{\varepsilon}^{n}\right)^{-}\right\rangle +\left(\overline{y}^{n-1}+\mathcal{H}(y_{\varepsilon}^{n})-\overline{\xi}^{n-1},\,\left(y_{\varepsilon}^{n}\right)^{-}\right).
\]
Since $u^{n}\geq0$ a.e. in $\Gamma_{C},$ $y^{n-1}\geq0$ a.e. in
$\Omega$ and $0\leq\xi^{n-1}\leq1$ a.e. in $\Omega$ and using the
fact that $\mathcal{H}_{\varepsilon}(x)=1$ for $x\leq0$ we obtain
\[
\left\langle A\left(y_{\varepsilon}^{n}\right)^{-},\,\left(y_{\varepsilon}^{n}\right)^{-}\right\rangle =\left\langle Bu^{n},\,\left(y_{\varepsilon}^{n}\right)^{-}\right\rangle +\left(\overline{y}^{n-1}+1-\overline{\xi}^{n-1},\,\left(y_{\varepsilon}^{n}\right)^{-}\right)\leq0.
\]
The coercivity of $A$ leads to $\left(y_{\varepsilon}^{n}\right)^{-}=0$
a.e. in $\Omega$ and then $y_{\varepsilon}^{n}\geq0$ a.e. in $\Omega.$
Consequently, the solution $y_{\varepsilon}^{n}$ is a solution of
$\left(\mathcal{WF}_{\varepsilon}^{n}\right).$ 
\end{proof}
Now for any $\varepsilon>0,$ let $y_{\varepsilon}^{n}$ be the solution
of the regularized problem $\left(\mathcal{WF}_{\varepsilon}^{n}\right).$
From \eqref{eq-Reg-Sol-Est} we can find a subsequence, also denoted
$(y_{\varepsilon}^{n})_{\varepsilon>0},$ such that 
\[
\begin{array}{rl}
y_{\varepsilon}^{n}\rightharpoonup y^{n} & \textrm{ in }H^{1}\left(\Omega\right),\\
y_{\varepsilon}^{n}\longrightarrow y^{n} & \textrm{ in }L^{2}\left(\Omega\right),\\
\mathcal{H}_{\varepsilon}(y_{\varepsilon}^{n})\overset{*}{\rightharpoonup}\xi^{n} & \textrm{ in }L^{\infty}\left(\Omega\right).
\end{array}
\]
By passing to the limit, we deduce that 
\begin{equation}
\left\{ \begin{array}{rcl}
 &  & y^{n}\geq0\qquad\textrm{ a.e. in }\Omega,\\
 &  & 0\leq\xi^{n}\leq1\qquad\textrm{ a.e. in }\Omega,\\
 &  & Ay^{n}=\xi^{n}+Bu^{n}+\overline{y}^{n-1}-\overline{\xi}^{n-1}\qquad\text{ in }\mathcal{V}^{\prime}.
\end{array}\right.
\end{equation}

Further, observe that 
\begin{equation}
y^{n}\geq0,\quad\xi^{n}\in\mathcal{H}(y^{n})\Longleftrightarrow y^{n}\geq0\qquad0\leq\xi^{n}\leq1\qquad\left(y^{n},\,\xi^{n}\right)=0.\label{HeavisideComplementarity}
\end{equation}
Therefore to complete the proof of existence of a solution for the
initial problem, it remains to prove that $\left(y^{n},\,\xi^{n}\right)=0.$
One has 
\begin{equation}
\left(y_{\varepsilon}^{n},\:\mathcal{H}_{\varepsilon}(y_{\varepsilon}^{n})\right)\longrightarrow\left(y^{n},\:\xi^{n}\right)\label{p(1-theta)-1}
\end{equation}
from the $L^{2}\left(\Omega\right)$ strong convergence of $y_{\varepsilon}^{n}$
to $y^{n}$ and the $L^{\infty}\left(\Omega\right)$ weak-{*} convergence
of $\mathcal{H}_{\varepsilon}(y_{\varepsilon}^{n})$ to $\xi^{n}.$

On the other hand, from $\mathcal{H}_{\varepsilon}$ expression we
have 
\[
\left(y_{\varepsilon}^{n},\:\mathcal{H}_{\varepsilon}(y_{\varepsilon}^{n})\right)\leq\varepsilon\text{meas}\left(\Omega\right)\longrightarrow0.
\]
Consequently $\left(y^{n},\:\xi^{n}\right)=0.$

Now, notice that if $\left(y^{n},\,\xi^{n}\right)\in\mathcal{V}\times L^{2}\left(\Omega\right)$
is a solution to the complementarity problem 
\[
\left(\mathcal{CS}^{n}\right)\left\{ \begin{array}{l}
\xi^{n}+Ay^{n}=Bu^{n}+\overline{y}^{n-1}-\overline{\xi}^{n-1}\qquad\text{ in }\mathcal{V}^{\prime},\\
y^{n}\geq0,\qquad\xi^{n}\geq0,\qquad\left(y^{n},\,\xi^{n}\right)=0\qquad\textrm{ a.e. in }\Omega,
\end{array}\right.
\]
then $y^{n}$ is a solution to the variational inequality 
\begin{equation}
\left(\mathcal{VI}^{n}\right)\left\{ \begin{array}{l}
y^{n}\in\mathcal{K}:=\left\{ q\in\mathcal{V}:\quad q\geq0\text{ a.e. in }\Omega\right\} ,\\
\left\langle Ay^{n},\,q-y^{n}\right\rangle \geq\left\langle Bu^{n},\,q-y^{n}\right\rangle +\left(\overline{y}^{n-1}-\overline{\xi}^{n-1},\,q-y^{n}\right)\qquad\forall q\in\mathcal{K}.
\end{array}\right.\label{eq-VI}
\end{equation}

Since $\left(\mathcal{VI}^{n}\right)$ possesses a unique solution
in $\mathcal{V}$ by virtue of (Stampacchia - Rodriguez), we deduce
that $y^{n}$ is unique. \\
 Finally, the uniqueness of $\xi^{n}$ follows from the uniqueness
of $y^{n}.$ More precisely, if $\left(y^{n},\,\xi_{1}^{n}\right)\in\mathcal{V}\times L^{2}\left(\Omega\right)$
and $\left(y^{n},\,\xi_{2}^{n}\right)\in\mathcal{V}\times L^{2}\left(\Omega\right)$
are two solutions to $\left(\mathcal{CS}^{n}\right)$ then 
\[
\xi_{1}^{n}=\xi_{2}^{n}=Ay^{n}-Bu^{n}-\overline{y}^{n-1}+\overline{\xi}^{n-1}\qquad\text{in }\mathcal{V}^{\prime}.
\]
Therefore $\xi_{1}^{n}-\xi_{2}^{n}=0$ in $\mathcal{V}^{\prime}.$
By the density of $\mathcal{V}\supset H_{0}^{1}\left(\Omega\right)$
in $L^{2}\left(\Omega\right)$ we conclude that $\xi_{1}^{n}=\xi_{2}^{n}$
in $L^{2}\left(\Omega\right).$

\section*{Appendix B: mathematical optimization in Banach spaces}

Let $\mathcal{X}$ and $\mathcal{Y}$ be real Banach spaces. For 
\begin{eqnarray*}
F:\,\mathcal{X} & \longrightarrow & \mathbb{R}\quad\text{Frechet-differentiable functional,}\\
g:\,\mathcal{X} & \longrightarrow & \mathcal{Y}\quad\text{continuously Frechet-differentiable,}
\end{eqnarray*}
we consider the following mathematical program: 
\begin{equation}
\min\left\{ F\left(x\right)\;|\;g\left(x\right)\in M,\;x\in C\right\} ,\label{eq:BanachMathProg}
\end{equation}
where $C$ is a closed convex subset of $\mathcal{X}$ and $M$ a
closed cone in $\mathcal{Y}$ with vertex at $0.$\\
 We suppose that the problem $\left(\ref{eq:BanachMathProg}\right)$
has an optimal solution $\hat{x,}$ and we introduce the conical hulls
of $C-\left\{ \hat{x}\right\} $ and $M-\left\{ y\right\} ,$ respectively,
by 
\begin{eqnarray*}
C\left(\hat{x}\right) & = & \left\{ x\in\mathcal{X}\;|\;\exists\beta\geq0,\,\exists c\in C,\;x=\beta\left(c-\hat{x}\right)\right\} ,\\
M\left(y\right) & = & \left\{ z\in\mathcal{Y}\;|\;\exists\lambda\geq0,\,\exists\zeta\in M,\;z=\zeta-\lambda y\right\} .
\end{eqnarray*}
The main result concerning the existence of a Lagrange multiplier
for $\left(\ref{eq:BanachMathProg}\right)$ is given in the next Theorem. 
\begin{thm*}
\label{ThmZowe}\cite{ZoweKurcyusz-1979} Let $\hat{x}$ be an optimal
solution of the problem $\left(\ref{eq:BanachMathProg}\right)$ satisfying
the following constraints qualification 
\begin{equation}
g^{\prime}\left(\hat{x}\right)\cdot C\left(\hat{x}\right)-M\left(g\left(\hat{x}\right)\right)=\mathcal{Y}.\label{eq:ZoweConstQualif}
\end{equation}
Then there exists a Lagrange multiplier $\mu^{*}\in\mathcal{Y}^{*}$
such that 
\begin{eqnarray}
\left\langle \mu^{*},\,z\right\rangle _{\mathcal{Y}^{*},\mathcal{Y}} & \geq & 0\quad\forall z\in M,\label{eq:ZoeLagMult1}\\
\left\langle \mu^{*},\,g\left(\hat{x}\right)\right\rangle _{\mathcal{Y}^{*},\mathcal{Y}} & = & 0,\label{eq:ZoeLagMult2}\\
F^{\prime}\left(\hat{x}\right)-\mu^{*}\circ g^{\prime}\left(\hat{x}\right) & \in & C\left(\hat{x}\right)_{+},\label{eq:ZoeLagMult3}
\end{eqnarray}
where $A_{+}=\left\{ x^{*}\in\mathcal{X}^{*}:\;\left\langle x^{*},\,a\right\rangle _{\mathcal{X}^{*},\mathcal{X}}\geq0\;\forall a\in A\right\} ,$
$\mathcal{Y}^{*}$ and $\mathcal{X}^{*}$ are the topological dual
spaces of $\mathcal{Y}$ and $\mathcal{X},$ respectively, and $(\mu^{*}\circ g^{\prime}\left(\hat{x}\right))\,d=\left\langle \mu^{*},\,g^{\prime}\left(\hat{x}\right)\,d\right\rangle _{\mathcal{Y}^{*},\mathcal{Y}}$
$\forall d\in\mathcal{X}.$ 
\end{thm*}

\end{document}